\documentclass{article}
\usepackage{fullpage}

\usepackage{lipsum}
\usepackage{amsfonts}
\usepackage{graphicx}
\usepackage{epstopdf}
\usepackage{algorithmic}
\usepackage{microtype}
\usepackage{booktabs}
\usepackage{hyperref}
\usepackage{algorithm}

\usepackage{amsmath,amssymb,amsthm}
\usepackage{amsopn}
\usepackage[erewhon]{newtxmath}
\usepackage{bbold}
\usepackage[mathscr]{eucal}
\usepackage{framed,url}
\usepackage{color}
\usepackage{psfrag}
\usepackage{wrapfig}
\usepackage{graphicx,subcaption,tikz,epstopdf}
\usepackage{enumitem,array}
\usepackage{xfrac}
\usepackage{verbatim}
\usepackage{textcomp}
\usepackage{balance}
\usepackage{lipsum}
\usepackage{multirow,multicol}
\usepackage{cite}
\ifpdf
  \DeclareGraphicsExtensions{.eps,.pdf,.png,.jpg}
\else
  \DeclareGraphicsExtensions{.eps}
\fi

\newtheorem{theorem}{Theorem}

\newtheorem{lemma}{Lemma}

\theoremstyle{definition}
\newtheorem{definition}{Definition}

\newcommand{\bfu}{\boldsymbol{u}}
\newcommand{\bfv}{\boldsymbol{v}}

\newcommand{\bfx}{\boldsymbol{x}}
\newcommand{\bfy}{\boldsymbol{y}}

\newcommand{\bfA}{\boldsymbol{A}}
\newcommand{\bfB}{\boldsymbol{B}}

\newcommand{\bfG}{\boldsymbol{G}}

\newcommand{\bfI}{\boldsymbol{I}}

\newcommand{\bfP}{\boldsymbol{P}}
\newcommand{\bfQ}{\boldsymbol{Q}}

\newcommand{\bfU}{\boldsymbol{U}}
\newcommand{\bfV}{\boldsymbol{V}}

\newcommand{\bfX}{\boldsymbol{X}}
\newcommand{\bfY}{\boldsymbol{Y}}

\newcommand{\bftheta}{\boldsymbol{\theta}}

\newcommand{\bbR}{\mathbb{R}}

\newcommand{\cI}{\mathcal{I}}

\newcommand{\cU}{\mathcal{U}}

\newcommand{\tX}{\boldsymbol{\mathscr{X}}}
\newcommand{\tY}{\boldsymbol{\mathscr{Y}}}
\newcommand{\tG}{\boldsymbol{\mathscr{G}}}
\newcommand{\tM}{\boldsymbol{\mathscr{M}}}

\newcommand{\T}{{\!\top\!}}

\DeclareMathOperator*{\minimize}{minimize} 
\DeclareMathOperator*{\maximize}{maximize} 

\DeclareMathOperator{\tr}{Tr}

\DeclareMathOperator{\Proj}{Proj}

\DeclareMathOperator{\nnz}{nnz}
\DeclareMathOperator{\qr}{QR}

\newcommand{\mode}[1]{{(\!#1\!)\!}}

\DeclareMathOperator{\St}{St}
\DeclareMathOperator{\grad}{grad}
\DeclareMathOperator{\NewProd}{TTMcTC}

\title{HOQRI: Higher-order QR Iteration for Low Multilinear Rank Approximation of Large and Sparse Tensors}

\author{Yuchen Sun\thanks{Department of Computer and Information Science and Engineering, University of Florida, Gainesville, FL 32611 USA.}
\and Amit Bhat\thanks{Department of Mathematics, University of Florida, Gainesville, FL 32611 USA.}
\and Chunmei Wang\footnotemark[2]
\and Kejun Huang\footnotemark[1]
}

\begin{document}

\maketitle

\begin{abstract}
We propose a new algorithm called higher-order QR iteration (HOQRI) for computing low multilinear rank approximation (LMLRA), also known as the Tucker decomposition, of large and sparse tensors. Compared to the celebrated higher-order orthogonal iterations (HOOI), HOQRI relies on a simple orthogonalization step in each iteration rather than a more sophisticated singular value decomposition step as in HOOI. More importantly, when dealing with extremely large and sparse data tensors, HOQRI completely eliminates the intermediate memory explosion by defining a new sparse tensor operation called TTMcTC (short for tensor times matrix chains times core). Furthermore, recognizing that the orthonormal constraints form a Cartesian product of Stiefel manifolds, we introduce the framework of manifold optimization and show that HOQRI guarantees convergence to the set of stationary points. Numerical experiments on synthetic and real data showcase the effectiveness of HOQRI. 
\end{abstract}


\section{Introduction}
A tensor is a data array indexed by three or more indices. It has been proven to be extremely useful in numerous applications \cite{kolda2009tensor,comon2014tensors,anandkumar2014tensor,cichocki2015tensor,papalexakis2016tensors,sidiropoulos2017tensor}. Two of the fundamental tensor factorization models are the canonical polyadic decomposition (CPD) \cite{hitchcock1927expression,harshman1970foundations,carroll1970analysis} and the Tucker decomposition \cite{hitchcock1927expression,tucker1966some}. In this paper, we propose a new algorithm called higher-order QR iteration (HOQRI) for Tucker decomposition, which is particularly suitable with large and sparse data tensors.

The Tucker decomposition is highly related to the higher-order SVD (HOSVD) and the best rank-$(K_1,\ldots, K_N)$ approximation \cite{de2000multilinear,de2000best}, which has found applications in computer vision \cite{vasilescu2003multilinear}, multilinear subspace learning \cite{lu2008mpca,lu2011survey}, topic modeling \cite{anandkumar2013overcomplete}, to name just a few. The state-of-the-art algorithm for (approximately) computing the Tucker decomposition remains the higher-order orthogonal iteration (HOOI) \cite{kroonenberg2008applied}. However, HOOI is not without limitations when being used in practice. One general issue that exists for tensors of all sizes is that it requires a reliable subroutine to calculate the singular value decomposition (SVD) of a matrix---it would be nice if the algorithm does not rely on that in general. More importantly, when dealing with extremely large but sparse tensor data, HOOI generates very large and dense intermediate matrices in every iteration. Sometimes the intermediate memory explosion exceeds the original data tensor, making the subsequent computations almost impossible to carry on \cite{kolda2008scalable,jeon2015haten2,oh2017s,zhang2020fast}. Last but not least, there is still limited understanding of the convergence of HOOI. The best known result states that every limit point of HOOI is a stationary point if some block non-degeneracy assumption is satisfied, which is hard to verify in practice \cite{xu2018convergence}.

We aim to address all three issues with the proposition of HOQRI. First of all, HOQRI defines a new tensor operation called $\NewProd$ (short for tensor times matrix chains times core) that entirely eliminates the intermediate memory explosion. Furthermore, each iteration of HOQRI involves only an orthogonalization step, which can be efficiently (and explicitly) calculated using the QR factorization (via, for example, Gram-Schmidt), rather than a sophisticated SVD step. As a sneak peek of how much the per-iteration complexity is improved using the proposed HOQRI algorithm, we compare the elapsed time per iteration with some well-known implementations in Figure~\ref{BaselinesDimension}. Finally, recognizing that the orthonormal constraints form a Cartesian product of Stiefel manifolds, we introduce the framework of manifold optimization and show that \mbox{HOQRI} guarantees convergence to the set of stationary points. Numerical experiments on synthetic and real data showcase the effectiveness of HOQRI.

\begin{figure}[t]
\centering
\includegraphics[width=1\linewidth]{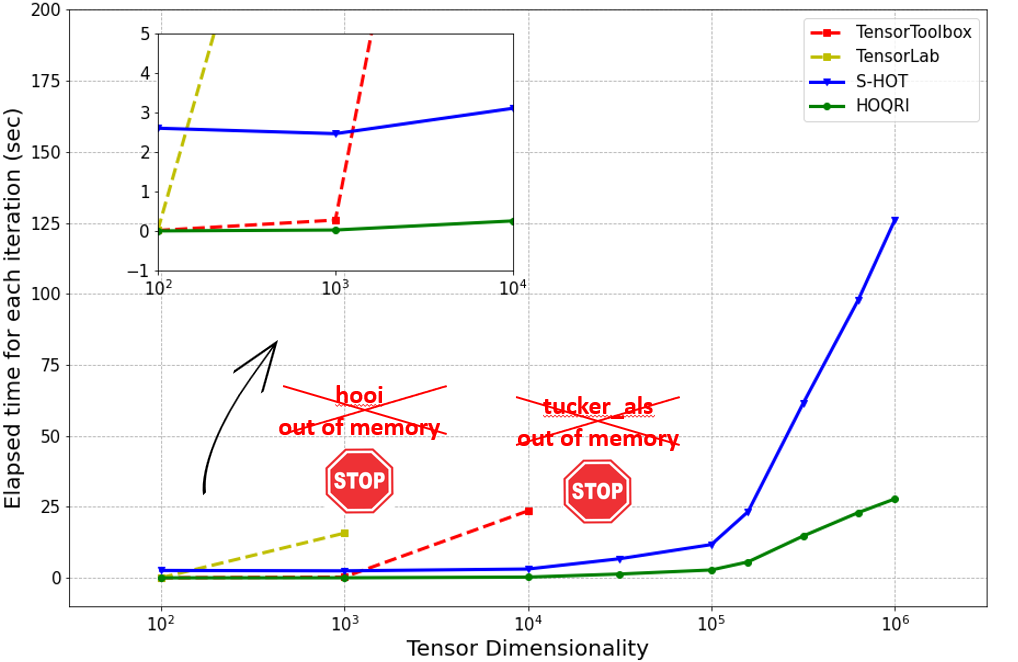}
\caption{Averaged per-iteration elapsed time of HOQRI comparing with Tensorlab \cite{tensorlab3.0}, Tensor Toolbox \cite{tensor_toolbox}, and S-HOT \cite{oh2017s} on randomly generated sparse tensors. The target multilinear rank for the Tucker decomposition is fixed at $10\times10\times10$, and the dimension of the tensor increases from $100\times100\times100$ to $10^6\times10^6\times10^6$. The number of nonzeros of each tensor is $10\times$ its number of rows.}
\label{BaselinesDimension}
\end{figure}

A preliminary version of the work has been published at ICASSP 2022 \cite{sun2022hoqri}. We made several improvements in this full-length journal version. First, the matrix version of implementing $\NewProd$ is slightly modified to save more per-iteration complexity. Second, we provide a novel convergence analysis of the algorithm based on the framework of manifold optimization. Our novel analysis shows that HOQRI is guaranteed to converge to a stationary point. Finally, we develop a C++ implementation to accompany the paper. In the conference version, we did not compare HOQRI with a baseline algorithm S-HOT \cite{oh2017s} since S-HOT was implemented in C++ as we did not consider it fair to compare with our preliminary MATLAB implementation; such promised comparison is fulfilled here to fully demonstrate the exceptional performance of HOQRI.

\section{Preliminaries}
\label{sec:prelim}

In this section, we define the notation used throughout the paper and also provide a brief overview of tensor factorization. For additional information about tensors and their factorizations, we direct the readers to the surveys by Kolda and Bader \cite{kolda2009tensor}, Papalexakis et al. \cite{papalexakis2017tensors}, and Sidiropoulos et al. \cite{sidiropoulos2017tensor}.

\subsection{Tensors and notations}

We denote the input $N$-way tensor, of size $I_1\times I_2\times\cdots\times I_N$, as $\tX$. In general, we denote tensors by boldface Euler script capital letters, e.g., $\tX$ and $\tY$, while matrices and vectors are denoted by boldface italic capital letters (e.g., $\bfX$ and $\bfY$) and boldface italic lowercase letters (e.g., $\bfx$ and $\bfy$), respectively. The Euclidean norm of a tensor $\tX$ is denoted as $\|\tX\|$, which is defined as
\[
\|\tX\| = \sqrt{\sum_{i_1=1}^{I_1}\cdots\sum_{i_N=1}^{I_N}\tX(i_1,...,i_N)^2}.
\]

\noindent
\textbf{Unfolding.}
A tensor can be unfolded, or \emph{matricized}, along any of its mode into a matrix. The tensor unfolding along the $n$th mode is denoted $\bfX_\mode{n}\in\bbR^{I_n\times\prod_{\nu\neq n}I_\nu}$. Simply put, the $n$th mode of $\tX$ forms the rows of $\bfX_\mode{n}$ and the remaining modes form the columns.

\noindent
\textbf{Tensor-matrix mode product.}
The mode-$n$ tensor-matrix product multiplies a tensor with a matrix along the $n$th mode. Suppose $\bfB$ is a $K\times I_n$ matrix, the $n$-mode tensor-matrix product, denoted as $\tX\times_n\bfB$, outputs a tensor of size $I_1\times\cdots\times I_{n-1}\times K\times I_{n+1}\times\cdots\times I_N$. Elementwise, 
\[
\big[\tX\times_n\bfB\big](i_1,...,i_{n\!-\!1},k,i_{n\!+\!1},...,i_N) = \sum_{i_n=1}^{I_N}\bfB(k,i_n)\tX(i_1,...,i_N).
\]
Using mode-$n$ unfolding, it can be equivalently written as
\[
\big[\tX\times_n\bfB\big]_\mode{n} = \bfB\bfX_\mode{n}.
\]
Note that the resulting tensor is in general dense regardless of the sparsity pattern of $\tX$.

A common task is to multiply a tensor by a set of matrices. This operation is called the tensor-times-matrix chain (TTMc). When multiplication is performed with all $N$ modes, it is denoted as \mbox{$\tX\times\{\bfB\}$}, where $\{\bfB\}$ is the set of $N$ matrices $\bfB^\mode{1},...,\bfB^\mode{N}$. Sometimes the multiplication is performed with all modes \emph{except one}. This is denoted as $\tX\times_{-n}\{\bfB\}$, where $n$ is the mode not being multiplied:
\[
\tX\times_{-n}\{\bfB\} = \tX\times_1\bfB^\mode{1}\cdots\times_{n-1}\bfB^\mode{n-1}\times_{n+1}\bfB^\mode{n+1}\cdots\times_N\bfB^\mode{N}.
\]

\noindent
\textbf{Kronecker product.}
The Kronecker product (KP) of $\bfA\in\bbR^{\ell\times m}$ and $\bfB\in\bbR^{p\times q}$, denoted as $\bfA\otimes\bfB$, is an $\ell p\times mq$ matrix defined as
\[
\bfA\otimes\bfB = \begin{bmatrix}
\bfA(1,1)\bfB & \cdots & \bfA(1,m)\bfB \\
\vdots & \ddots & \vdots \\
\bfA(\ell,1)\bfB & \cdots & \bfA(\ell,m)\bfB
\end{bmatrix}.
\]
Mathematically, the $n$-mode TTMc can be equivalently written as the product of mode-$n$ unfolding times a chain of Kronecker products:
\begin{align}\label{eq:naive}
&\big[\tX\times_{-n}\{\bfB\}\big]_\mode{n} \nonumber \\
&= \bfX_\mode{n} 
\left(\bfB^\mode{1}\otimes\cdots\otimes\bfB^\mode{n-1}\otimes\bfB^\mode{n+1}\otimes\cdots\otimes\bfB^\mode{N}\right)^\T.
\end{align}

More notations are shown in Table \ref{tab:notation}.
\begin{table}[t]
\centering
\caption{List of notations}\label{tab:notation}
\begin{tabular}{l|l}
	\hline
	Symbol & Definition \\
	\hline
	$N$		& number of modes \\
	$\tX$	& $N$-way data tensor of size $I_1\times I_2 \times \cdots \times I_N$\\
	$\tX(i_1,...,i_N)$ & $(i_1,...,i_N)$-th entry of $\tX$ \\
	$\cI(\tX)$ & index set of all nonzero entries in $\tX$ \\
	$\cI_{i_n}^\mode{n}(\tX)$ & subset of $\cI(\tX)$ where the $n$-mode index is $i_n$ \\
	$\bfX_\mode{n}$ & mode-$n$ matrix unfolding of $\tX$ \\
	$I_n$ & dimension of the $n$th mode of $\tX$ \\
	$K_n$ & multilinear rank of the $n$th mode\\
	$\tG$ & core tensor of the Tucker model $\in\bbR^{K_1\times...\times K_N}$\\
	$\bfU^\mode{n}$ & mode-$n$ factor of the Tucker model $\in\bbR^{I_n\times K_n}$\\
	$\{\bfU\}$ & set of all factors $\{\bfU^\mode{1},...,\bfU^\mode{N}\}$ \\
	$\times_{n}$ & $n$-mode tensor-matrix product \\
	$\times_{-n}$ & chain of mode products except the $n$th one \\
	\hline
\end{tabular}
\end{table}

\subsection{Tucker decomposition}\label{sec:tucker}

The goal of Tucker decomposition is to approximate a data tensor $\tX\in\bbR^{I_1\times\cdots\times I_N}$ with the product of a core tensor $\tG\in\bbR^{K_1\times\cdots\times K_N}$ and a set of $N$ factor matrices $\bfU^\mode{n}\in\bbR^{I_n\times K_n}, n=1,...,N$, i.e., $\tX\approx\tG\times\{\bfU\}$. An illustration of Tucker decomposition for 3-way tensors is shown in Figure~\ref{fig:tucker}.
To find the Tucker decomposition of a given tensor $\tX\in\bbR^{I_1\times\cdots\times I_N}$ and a target reduced dimension $K_1\times\cdots\times K_N$, one formulates the following problem:
\begin{equation}\label{prob:tucker}
\minimize_{\substack{\tG\in\bbR^{K_1\times\cdots\times K_N}\\ \{\bfU^\mode{n}\in\bbR^{I_n\times K_n}\}_{n=1}^N}}~~ 
\left\| \tX - \tG\times\{\bfU\}\right\|^2.
\end{equation}

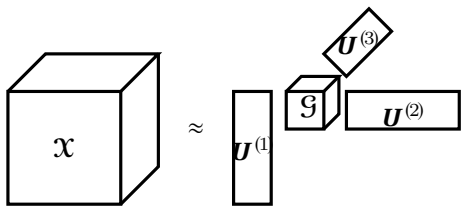
\begin{figure}[t]
	\centering
	\begin{tikzpicture}
\draw[very thick] (0,0) -- (0,1.5) -- (1.5,1.5) -- (1.5,0) -- cycle;
\draw[very thick] (0,1.5) -- (0.5,2) -- (2,2) -- (1.5,1.5);
\draw[very thick] (2,2) -- (2,0.5) -- (1.5,0);

\node at (0.75,0.75) {$\tX$};

\node at (2.5,1) {$\approx$};

\draw[very thick] (3,0) -- (3.5,0)-- (3.5,1.5) -- (3,1.5) -- cycle;
\node at (3.25,0.75) {$\bfU^\mode{1}$};

\draw[very thick] (3.7,1) -- (4.2,1) -- (4.2,1.5) -- (3.7,1.5) -- cycle;
\draw[very thick] (3.7,1.5) -- (3.9,1.7) -- (4.4,1.7) -- (4.2,1.5);
\draw[very thick] (4.4,1.7) -- (4.4,1.2) -- (4.2,1);
\node at (4,1.3) {$\tG$};

\draw[very thick] (4.5,1) -- (6,1) -- (6,1.5) -- (4.5,1.5) -- cycle;
\node at (5.25,1.2) {$\bfU^\mode{2}$};

\draw[very thick] (4.5,1.7) -- (4.2,2) -- (4.8,2.6) -- (5.1,2.3) -- cycle;
\node at (4.65,2.15) {$\bfU^\mode{3}$};
\end{tikzpicture}
	\caption{Tucker decomposition of a 3-way tensor}
	\label{fig:tucker}
\end{figure}

Because of the rotational ambiguities, one can impose the constraints that the columns of each $\bfU^\mode{n}$ are orthonormal without loss of generality, leading to the formulation:
\begin{align*}
\minimize_{\tG,\{\bfU\}}~~& \left\| \tX - \tG\times\{\bfU\} \right\|^2 \\
\textup{subject to}~~ & ~\bfU^{\mode{n}\,\T}\bfU^\mode{n} = \bfI, n=1,...,N.
\end{align*}
It can be shown that the optimal $\tG$ should be taken as 
\(
\tG = \tX \times \{\bfU^\T\},
\)
and plugging it back leads to the simplified formulation~\cite{de2000best} by eliminating the variable $\tG$:
\begin{equation}\label{prob:main}
\begin{aligned}
\maximize_{\{\bfU\}}~~~& \left\| \tX \times \left\{\bfU^\T\right\} \right\|^2 \\
\textup{subject to}~~& ~ \bfU^{\mode{n}\,\T}\bfU^\mode{n} = \bfI, n=1,...,N.
\end{aligned}
\end{equation}

A well-known algorithmic framework for approximately solving~\eqref{prob:main} is the higher-order orthogonal iteration (HOOI) \cite{kroonenberg2008applied,de2000best}, which takes a block coordinate descent approach and updates each $\bfU^\mode{n}$ in a cyclic fashion. Fixing all variables except $\bfU^\mode{n}$, it is shown that a conditionally optimal update for $\bfU^\mode{n}$ can be obtained by taking the $K_n$ leading left singular vectors of $\bfY_\mode{n}$, which is obtained by taking the $n$-mode unfolding of the tensor $\tY\triangleq\tX\times_{-n}\{\bfU^\T\}$. A detailed description of HOOI is given in Algorithm~\ref{alg:hooi}. Solving~\eqref{prob:main} exactly is NP-hard~\cite{hillar2013most}, but the HOOI algorithm is guaranteed to monotonically increase the objective of~\eqref{prob:main}, and in a lot of cases finds a good Tucker decomposition. When computationally viable, HOOI is usually initialized with the leading components of the higher-order SVD (HOSVD) of $\tX$~\cite{de2000multilinear}. A variant of HOOI is proposed in \cite{chen2015optimal} using the maximum block improvement (MBI) framework: instead of updating the $N$ factors in a cyclic fashion, MBI updates the one that leads to the most improvement. The benefit is that it guarantees convergence to a stationary point of \eqref{prob:main}, at the cost of a lot of additional computations in order to find the best factor at every step. Recently, Xu \cite{xu2018convergence} provides a more comprehensive analysis of the original HOOI algorithm, which we explain in detail in Section~\ref{sec:statement}.

\begin{algorithm}[t]
	\caption{Higher-Order Orthogonal Iteration (HOOI)}
	\label{alg:hooi}
	\begin{algorithmic}[1]
		\STATE initialize all $\bfU^\mode{n}$
		\COMMENT{randomly or from HOSVD}
		\REPEAT
		\FOR{$n=1,...,N$}
		\STATE $\tY~~\;\leftarrow\tX\times_{-n}\{\bfU^\T\}$
		\STATE $\bfU^\mode{n}\leftarrow$ $K_n$ leading left singular vectors of $\bfY_\mode{n}$
		\ENDFOR
		\UNTIL{convergence}
		\STATE $\tG\leftarrow\tX\times\{\bfU^\T\}$
	\end{algorithmic}
\end{algorithm}

Several other algorithms exist in addition to HOOI and its variants. Some directly tackles the fitting problem \eqref{prob:tucker} without imposing the orthonormal constraints, for example by using alternating least squares \cite{jeon2015haten2}, gradient descent \cite{Traore2019Singleshot}, or nonlinear least squares \cite{sorber2015structured}.
On the other hand, several algorithms already recognized that the orthogonality constrained problem \eqref{prob:main} falls into the framework of manifold optimization, so existing manifold optimization algorithms can be applied. Most of them focus on manifold variants of second-order algorithms, such as Newton's method \cite{ishteva2009differential,elden2009newton}, quasi-Newton algorithm \cite{savas2010quasi}, trust-region method \cite{ishteva2011best}, and nonlinearly preconditioned solvers \cite{de2016nonlinearly}. Some of them do provide convergence guarantee to a stationary point. However, these second-order algorithms are very hard to scale up to handle large and sparse tensor data that we commonly face nowadays.
Finally, to specifically address the large-scale problem, strategies such as stochastic gradient descent \cite{li2020sgd} and sketching \cite{malik2018low} have also been proposed; these approximation algorithms are beyond the scope of this work and therefore will not be considered further.

\section{Problem Statement}\label{sec:statement}

There are several computational issues regarding HOOI, especially when the data tensor $\tX$ is large and sparse. The immediate issue people noticed is that a naive implementation to calculate $\tX\times_{-n}\{\bfU^\T\}$ may require an enormous amount of memory due to intermediate dense data storage according to the explicit matrix product \eqref{eq:naive}. This issue is somewhat resolved by exploiting the sparsity structure of the data tensor ensuring that the intermediate memory load is no more than the maximum of $\tX$ and any of the $\tY$'s~\cite{kolda2008scalable,jeon2015haten2,smith2017tucker}. 

As the order of the tensor increases nowadays, it has been observed that even the amount of memory to record $\tY$ may not be practical for scalable Tucker decomposition. One way is to avoid explicitly materializing $\tY$ while still trying to find its leading singular vectors, as suggested by the S-HOT approach~\cite{oh2017s}. All the previous works pertain to the framework of HOOI, which requires computing the singular value decomposition (SVD) of some large and dense matrices as a sub-routine. This may be an issue as well, as there is not a single best way to implement SVD for all cases~\cite{golub1996matrix}, and may require inner iterations to compute the dominant singular vectors. The memory and computation complexities of these Tucker decomposition algorithms are compared in Table \ref{memory complexity}.

\begin{table}
\centering
\caption{Summary of per-iteration memory/computation complexities of various scalable Tucker decomposition algorithms. For simplicity, we assume $I_n=I$ and $K_n=K$ for all $n$. We calculate the \emph{intermediate} memory requirement for all methods, assuming there are already spaces allocated for the data $\tX$ and the solution $\tG$ and $\{\bfU\}$, which takes space $\nnz(\tX)+NIK+K^N$.}
\label{memory complexity}
\begin{tabular}{l|cc}
\hline
		& memory & computation \\
\hline
\\[-5pt]
Kronecker product \eqref{eq:naive} & $I^{N-1}K^{N-1}$	&  $I^NK^{N-1}+IK^N$ \\
MET \cite{kolda2008scalable} / TTMc \cite{smith2017tucker} & $IK^{N-1}$	& $\nnz(\tX)K^{N-1}+IK^N$\\
S-HOT \cite{oh2017s} & $K^{N-1}$ 	& $\nnz(\tX)K^{N-1}+IK^N$ \\
HOQRI (this paper) & $IK$ & $\nnz(\tX)K^N$ or\\
 & & $\nnz(\tX)K^{N-1}+IK^N$ \\
\hline
\end{tabular}

\end{table}

In terms of convergence, the best one has claimed for HOOI is a limited result that a limit point is a stationary point if it is block non-degenerate \cite{xu2018convergence}. Since the problem formulation \eqref{prob:main} is nonconvex, it is almost impossible to claim convergence to a global optimum, but it is often expected that an algorithm converges to a stationary point. This has unfortunately not been established for the case of HOOI, even though it falls into the general framework of block coordinate descent (BCD). It is shown that BCD converges to a stationary point if 1) each block of variables is constrained in a \emph{convex} set, and 2) each block update is resulted from a \emph{unique} minimizer \cite[Proposition 3.7.1]{bertsekas2016nonlinear}, neither of which is satisfied by HOOI. Even if we abandon the orthonormal constraint to make the constraint sets convex, convergence still cannot be guaranteed since the second requirement of a unique minimizer is an impractical assumption. Xu \cite{xu2018convergence} managed to show that every limit point of HOOI is a stationary point, but under the assumption that the limit points are ``nondegenerate'', meaning that the $K_n$th and $K_{n+1}$th singular values of $\bfY_\mode{n}$ are distinct for all $n=1,\ldots,N$. Intuitively, this requirement is not very strict, but it also means that it is possible that HOOI converges to a degenerate point, which is almost certainly not a stationary point.

\section{Proposed Algorithm: HOQRI}

In this paper, we propose a new algorithmic framework for Tucker decomposition called higher-order QR iteration (HOQRI), which resolves the three issues we mentioned for HOOI: It is free from calling SVD, and its memory requirement is minimal (no larger than that of storing the solution $\{\bfU\}$ and $\tG$). In the next section, we also show that HOQRI converges to the set of stationary points of \eqref{prob:main}, which has not been established for HOOI.

\subsection{Algorithm description}

An implementation of the proposed HOQRI is shown in Algorithm~\ref{alg:hoqri}. As we can see, it is a very simple algorithm with a few but very important modifications to the existing HOOI. The only ambiguous part of Algorithm~\ref{alg:hoqri} is in line 4, where we define a new tensor-matrix operation called tensor times matrix chain times core (TTMcTC). Mathematically, it is simply the product of TTMc \cite{smith2017tucker} and the corresponding matricized core-tensor:
\begin{equation}\label{eq:ttmctc}
\bfA^\mode{n} = \bfY_\mode{n}\bfG_\mode{n}^\T,
\end{equation}
where $\bfY_\mode{n}$ is the mode-$n$ matricization of the tensor $\tY=\tX\times_{-n}\{\bfU^\T\}$ as in line 4 of HOOI, and $\bfG_\mode{n}$ is the mode-$n$ matricization of the core tensor $\tG$. The reason this operation is described as a subroutine in Algorithm~\ref{alg:hoqri} is that the intermediate $\tY$ is never explicitly materialized when the data tensor is large and sparse---the result is directly calculated without additional memory overheads as will be explained in the sequel. This is the key step that renders the huge amount of savings in memory, thus making it most scalable.

In terms of per-iteration complexity, it is clear that it is dominated by the novel $\NewProd$ kernel. As will be explained in detail next, depending on the specific implementation, the complexity of $\NewProd$ is either $O(\nnz(\tX)K_1\cdots K_N)$ or $O(\nnz(\tX)K_1\cdots K_N/K_n+I_nK_1\cdots K_N)$. Although it does not seem a big improvement compared to its predecessors, in practice it is much faster thanks to the much smaller constant.
However, the biggest benefit is that it does not require additional memory to store the large and dense tensor $\tY$. Moreover, HOOI requires calculating the $K_n$ leading singular vectors of $\bfY_\mode{n}$ , which in general requires $O(I_nK_1\cdots K_N)$ complexity with a big constant in front, while HOQRI follows with a simple QR factorization of $\bfA^\mode{n}$ of size $I_n\times K_n$ with $O(I_nK_n^2)$ flops, which is negligible compared to $\NewProd$.

\begin{algorithm}[t]
\caption{Higher-Order QR Iteration (HOQRI)}
\label{alg:hoqri}
\begin{algorithmic}[1]
\STATE initialize all $\bfU^\mode{n}$
\COMMENT{randomly or from HOSVD}
\REPEAT
\FOR{$n=1,...,N$}
\STATE $\bfA^\mode{n}\leftarrow\NewProd(\tX,\{\bfU\},n)$
\STATE $\bfU^\mode{n}\leftarrow$ an orthonormal basis of $\bfA^\mode{n}$ via $\qr$ 
\ENDFOR
\UNTIL{convergence}
\STATE $\tG\leftarrow\tX\times\{\bfU^\T\}$
\end{algorithmic}
\end{algorithm}

\subsection{The TTMcTC kernel}

In Algorithms~\ref{alg:ttmctc1} and \ref{alg:ttmctc2}, we provide two detailed implementations of $\NewProd$ that are efficient for a large and sparse data tensor $\tX$. The first one given in Algorithm~\ref{alg:ttmctc1} is more straight-forward as it computes $\bfA^\mode{n}=\bfY_\mode{n}\bfG_\mode{n}^\T$ element by element. Specifically, the very first step of Algorithm~\ref{alg:ttmctc1} calculates $\tG$, which can be done with $O(\nnz(\tX)K_1\cdots K_N)$ flops. Then each element of $\bfA^\mode{n}$ is calculated as the inner product of the $i_n$th row of $\bfY_\mode{n}$ and $k_n$th row of $\bfG_\mode{n}$, but that row of $\bfY_\mode{n}$ is calculated on-the-fly without excessive memory overheads. All index $i_n$ and $k_n$, $n=1,\ldots,N$, is traversed exactly once, leading to another $O(\nnz(\tX)K_1\cdots K_N)$ complexity. Therefore, the overall complexity of Algorithm~\ref{alg:ttmctc1} is $O(\nnz(\tX)K_1\cdots K_N)$. This implementation is more appropriate for a lower-level language such as C++, in which case multiple for-loops do not slow down the execution. 

We also propose a different implementation in Algorithm~\ref{alg:ttmctc2} that involves fewer for-loops, which makes it more suitable for the MATLAB environment when matrix operations are much more efficient than explicit for-loops. We first notice that $\bfG_\mode{n}=(\bfU^\mode{n})^\T\bfY_\mode{n}$, therefore
\[
\bfA^\mode{n}=\bfY_\mode{n}\bfG_\mode{n}^\T
=\bfY_\mode{n}\bfY_\mode{n}^\T\bfU^\mode{n}
=\sum_{k}\bfy_k\bfy_k^\T\bfU^\mode{n},
\]
where $\bfy_k$ denotes the $k$th column of $\bfY_\mode{n}$. This allows us to calculate one column of $\bfY_\mode{n}$ at a time and form $\bfA^\mode{n}$ in a recursive manner. Computing each $\bfy_k$ is the mode-product of different columns of $\{\bfU\}$, so the complexity is $O(\nnz(\tX))$. The result is a vector of length $I_n$, so forming $\bfy_k\bfy_k^\T\bfU^\mode{n}$ requires $O(I_nK_n)$ flops. Summing over all $k_{-n}$ leads to the $O(\nnz(\tX)K_1\cdots K_N/K_n+I_nK_1\cdots K_N)$ complexity.

At first glance, it might come as somewhat unexpected that the complexity of Algorithms~\ref{alg:ttmctc1} and \ref{alg:ttmctc2} are different, but it is entirely reasonable. Both implementations tries to calculate $\bfA^\mode{n}
=\bfY_\mode{n}\bfY_\mode{n}^\T\bfU^\mode{n}$, and we know that even for elementary matrix multiplication, the complexity may be different depending on which of the two matrices are multiplied first. Algorithm~\ref{alg:ttmctc1} obviously calculates $\bfG_\mode{n}^\T=\bfY_\mode{n}^\T\bfU^\mode{n}$ first followed by $\bfA^\mode{n}=\bfY_\mode{n}\bfG_\mode{n}^\T$. Algorithm~\ref{alg:ttmctc2}, on the other hand, is essentially calculating $\bfY_\mode{n}\bfY_\mode{n}^\T$ first, which leads to the difference in complexities. In practice, the programming environment plays a much more important role in determining which one is more efficient.

\begin{algorithm}[t]
\caption{$\NewProd(\tX,\{\bfU\},n)$ ~~ element-wise}
\label{alg:ttmctc1}
\begin{algorithmic}[1]
\STATE $\tG\leftarrow\tX\times\{\bfU^\T\}$
\FOR{$i_n=1,...,I_n$}
\FOR{$k_n=1,...,K_n$}
\STATE \(\displaystyle
\bfA^\mode{n}(i_n,k_n)=\hspace*{-5pt}
\sum^{\textcolor{white}{I}}_{i_{\!-\!n}\in\cI_{i_n}^\mode{n}(\tX)}
\sum^{\textcolor{white}{K}}_{k_{\!-\!n}}
\tX(i_1,...,i_N)\tG(k_1,...,k_N)
 \times \prod_{\nu\neq n}\bfU^\mode{\nu}(i_\nu,k_\nu) \)
\ENDFOR
\ENDFOR
\RETURN $\bfA^\mode{n}$
\end{algorithmic}
\end{algorithm}

\begin{algorithm}[t]
\caption{$\NewProd(\tX,\{\bfU\},n)$ ~~ matrix version}
\label{alg:ttmctc2}
\begin{algorithmic}[1]
\STATE $\bfA^\mode{n}\leftarrow0$
\FOR{$k_{-n}$}
\STATE $\bfy = \tX\times_{-n}\{\bfu_{k_1},\ldots,\bfu_{k_N}\}$
\COMMENT{sparse mode product}
\STATE $\bfA^\mode{n}\leftarrow\bfA^\mode{n} + \bfy\bfy^\T\bfU^\mode{n}$\\
\ENDFOR
\RETURN $\bfA^\mode{n}$
\end{algorithmic}
\end{algorithm}

\subsection{Intuition}
Before we move on to prove convergence of HOQRI, we briefly describe the intuition behind the design of HOQRI here. As we explained in the introduction section, a preliminary version of the algorithm has appeared \cite{sun2022hoqri}, which is not exactly the same as Alg.~\ref{alg:hoqri} proposed in this paper. We will explain the difference and the reason behind the modification.

In line 5 of HOOI, we are supposed to calculate the $K_n$ leading singular vectors of $\bfY_\mode{n}$, or equivalently, the $K_n$ leading eigenvectors of $\bfY_\mode{n}\bfY_\mode{n}^\T$. One way of computing it is to apply the orthogonal iteration \cite[\S8.2.4]{golub1996matrix} which iteratively calculates the matrix product $\bfY_\mode{n}\bfY_\mode{n}^\T\bfU^\mode{n}$ and then orthogonalized the columns via the QR factorization. Notice that $\bfY_\mode{n}^\T\bfU^\mode{n}$ effectively calculates the mode-$n$ matricization of the core tensor $\tG$. On the other hand, instead of applying it multiple times to converge to the $K_n$ leading eigenvectors, we propose to calculate it only once per iteration, using the factor $\bfU^\mode{n}$ obtained from the previous iteration. As we will show next, this seemingly ad hoc modification not only greatly reduces the memory and computational overhead in each iteration, but also guarantees of convergence to a stationary point.

\section{Convergence Analysis}

In this section we provide a convergence analysis of HOQRI from the perspective of manifold optimization, since the orthogonality constraint is a smooth manifold called the Stiefel manifold, and the full set of constraint for the Tucker formulation \eqref{prob:main} is a Cartesian product of several Steifel manifolds, which remains a manifold. After showing that HOQRI guarantees monotonic improvement of the objective value, we briefly introduce the basic concepts in manifold optimization---for more details, cf. a standard textbook on the subject, e.g., \cite{absil2009optimization,boumal2020introduction}. HOQRI does not fall into the framework of any classical manifold optimization algorithm such as the Riemannian gradient descent, yet thanks to the special structure of the Tucker decomposition problem, we can show that it converges to a stationary point, meaning the Riemannian gradient goes to zero.

Throughout this section, we formally define the iterates generated by HOQRI as 
\[
\{\bfU_t\}=\{\bfU^\mode{1}_{t},\ldots,\bfU^\mode{N}_{t}\};
\]
at iteration $t$, only the $n$th factor $\bfU^\mode{n}$, where $n=(t\!\!\mod N)+1$, is updated, while other factor matrices remain the same, i.e.,
\begin{equation}\label{eq:iteration}
\bfU^\mode{n}_t=\begin{cases}
\text{QR of } \bfA^\mode{n}_{t-1}, &n=(t\!\!\!\mod N)+1, \\
\bfU^\mode{n}_{t-1}, & \text{otherwise},
\end{cases}
\end{equation}
where $\bfA^\mode{n}_{t-1}$ is obtained from $\NewProd(\tX,\{\bfU_{t-1}\},n)$. We also denote the objective function of \eqref{prob:main} as $f(\{\bfU\})$.

\subsection{Objective convergence}
We start by showing the simpler result that HOQRI guarantees that the objective function \eqref{prob:main} is monotonically nondecreasing, similar to what has been established to the celebrated HOOI \cite{de2000best}---this does not necessarily mean that the factor matrices $\{\bfU\}$ are guaranteed to converge to a stationary point. At first glance, it may not seem apparent that HOQRI also guarantees this property, which we show in the sequel.

\begin{theorem}\label{thm:obj}
The iterates generated from the HOQRI algorithm \eqref{alg:hoqri}, $\{\bfU_t\}$, guarantees that
\[
f(\{\bfU_{t+1}\})\geq f(\{\bfU_t\}),\quad\text{for all~} t=1,2,\ldots.
\] 
\end{theorem}

\begin{proof}
We prove it via the following process:
\begin{subequations}\label{eq:thm1}
\begin{align}
f(\{\bfU_t\})&=\hat{f}_t(\bfU_t^\mode{n}) \label{eq:thm1a}\\
&\leq \hat{f}_t(\bfQ_t\bfV_t^\T) \label{eq:thm1b}\\
&\leq \|\bfY_{\mode{n}\,t}^\T\bfQ_t\bfV_t^\T\|^2 \label{eq:thm1c}\\
&= f(\{\bfU_{t+1}\}), \label{eq:thm1d}
\end{align}
\end{subequations}
where the definitions of $\hat{f}$ and $\bfQ_t\bfV_t^\T$ will be given in the sequel.

At iteration $t+1$, when fixing all the other factors $\bfU^\mode{1}_t,\ldots,\bfU^\mode{n-1}_t$, $\bfU^\mode{n+1}_t,\ldots,\bfU^\mode{N}_t$, the subproblem
\begin{align}\label{eq:sub}
\maximize_{\bfU}~ \left\| \bfU^\T\bfY_{\mode{n}\,t} \right\|^2 \qquad \textup{subject to} ~ \bfU^\T\bfU = \bfI
\end{align}
is a convex quadratic maximization problem, where $\bfY_{\mode{n}\,t}$ is the mode-$n$ matrix unfolding of $\tX\times_{-n}\{\bfU_t\}$. For convex quadratics, the first-order Taylor approximation gives a global lower bound of the objective function, i.e.,
\begin{align}\label{eq:lower}
\left\| \bfU^\T\bfY_{\mode{n}\,t} \right\|^2 &\geq 
\left\| \bfU_t^{\mode{n}\,\T}\bfY_{\mode{n}\,t} \right\|^2 +
2\tr\left(\bfU_t^{\mode{n}\,\T}\bfY_{\mode{n}\,t}\bfY_{\mode{n}\,t}^\T
(\bfU-\bfU_t^\mode{n}) \right) \nonumber\\
&= \left\| \bfU_t^{\mode{n}\,\T}\bfY_{\mode{n}\,t} \right\|^2 +
2\tr\left(\bfA_t^{\mode{n}\,\T}(\bfU-\bfU_t^\mode{n}) \right).
\end{align}
Define the right-hand-side of \eqref{eq:lower} as $\hat{f}_t(\bfU)$, then indeed we have
\[
f(\{\bfU_t\}) = \left\| \bfU_t^{\mode{n}\,\T}\bfY_{\mode{n}\,t} \right\|^2 
= \hat{f}(\bfU^\mode{n}_t),
\]
which shows \eqref{eq:thm1a}, and
\[
\hat{f}(\bfU)\leq \left\| \bfU^\T\bfY_{\mode{n}\,t} \right\|^2, 
\]
for all $\bfU$, which shows \eqref{eq:thm1c} with $\bfU=\bfQ_t\bfV_t^\T$.

Maximizing $\hat{f}(\bfU)$ subject to orthonormal constraint is equivalent to
\begin{align*}
\maximize_{\bfU}~ \tr\bfA_t^{\mode{n}\,\T}\bfU \qquad \textup{subject to} ~ \bfU^\T\bfU = \bfI.
\end{align*}
The solution is given by the Procrustes rotation $\bfQ_t\bfV_t^\T$ \cite{schonemann1966generalized}, where $\bfQ_t$ and $\bfV_t$ are the left and right singular vectors of $\bfA_t^\mode{n}$. This shows \eqref{eq:thm1b}.

Finally, we recognize that both $\bfQ_t\bfV_t^\T$ and $\bfU^\mode{n}_{t+1}$ are orthonormal bases for the columns of $\bfA_t^\mode{n}$, thus they span the same subspace. As a result, 
\[
\|\bfY_{\mode{n}\,t}^\T\bfQ_t\bfV_t^\T\|^2 = \|\bfY_{\mode{n}\,t}^\T\bfU^\mode{n}_{t+1}\|^2 = f(\{\bfU_{t+1}\}).
\]
This shows \eqref{eq:thm1d}.
\end{proof}

\subsection{Primer on manifold optimization}

A manifold constrained optimization problem, as the name suggests, refers to a constrained optimization problem whose constraints are manifolds. Consider the general problem:
\begin{equation}\label{manifoldOptimization}
\maximize_{\bftheta\in \tM}~~f(\bftheta)
\end{equation}
where $\tM$ is a manifold constraint.
Our main objective function \eqref{prob:tucker} could be considered as a manifold optimization problem on the Cartesian product of $N$ Stiefel manifold constraints.
\begin{equation}\label{stiefelObjective}
\begin{aligned}
\maximize_{\{\bfU\}}~~ & \left\| \tX \times \left\{\bfU^\T\right\} \right\|^2 \\
\textup{subject to}~~ & \bfU^\mode{n}\in\St(I_n,K_n),~~ n=1,\ldots,N.
\end{aligned}
\end{equation}
The Stiefel manifold $\St(I,K)$ is a manifold in $\bbR^{I\times K}$ defined as the set of all matrices with orthonormal columns
\begin{equation}\label{eq:stiefel}
    \St(I,K)=\{ \bfU\in\bbR^{I\times K}:\bfU^\T\bfU=\bfI\}.
\end{equation}
In the remainder of this subsection, we will introduce some basic concepts in manifold optimization, and use the Stiefel manifold as the example to instantiate some key concepts such as tangent spaces and Riemannian gradients.

Informally, a manifold is a topological space that locally resembles Euclidean space near each point. For the purpose of this paper, we only focus on a special case called \emph{embedded submanifolds of linear spaces}, which is a subset of the Euclidean space $\bbR^d$ defined as $\tM=\{\bftheta~|~h(\bftheta)=0\}$ where $h$ is a set of smooth functions such that its Jacobian $Dh(\bftheta)$ has a fixed rank $k$ everywhere---the dimension of the manifold is therefore $d-k$. From this definition, we see that the Stiefel manifold $\St(I,K)$ is indeed a submanifold that is embedded in $\bbR^{I\times K}$, and its dimension is $IK-K(K+1)/2$.

A linear approximation of $\tM$ at the point $\bftheta$ is called a \emph{tangent space}, denoted as $\text{T}_{\bftheta}\tM$. If the manifold $\tM$ is defined as $\{\bftheta~|~h(\bftheta)=0\}$, then the tangent space is the null space of $Dh(\bftheta)$, the Jacobian matrix of $h$ at $\bftheta$. For Stiefel manifold, the tangent space is
\[
\text{T}_{\bfU}\St(I,K)=\{\bfV\in\bbR^{I\times K}: \bfU^\T\bfV+\bfV^\T\bfU=0\}.
\]
For an arbitrary point $\bfA\in\bbR^{I\times K}$, we will need to project it onto the tangent space $\text{T}_{\bfU}\St(I,K)$---which is not hard to derive as the tangent space is linear---denoted as
\begin{equation}\label{eq:proj}
\Proj_{\bfU}(\bfA)=\bfA - \bfU(\bfU^\T\bfA+\bfA^\T\bfU)/2.
\end{equation}
Notice that here we are projecting $\bfA$ onto the \emph{tangent space} of the manifold $\tM$, not the manifold $\tM$ itself. A commonly used operation that brings a point outside of the manifold onto \emph{the manifold} is called \emph{retraction}; however, it is limited to a point that lie on the tangent space of a point from the manifold. For a point $\bfU\in\St(I,K)$ and a direction $\bfV\in\text{T}_{\bfU}\St(I,K)$, one way to ``retract'' a point $\bfU+\alpha\bfV$ back to the Stiefel manifold is to apply the `thin' QR factorization to the matrix $\bfU+\alpha\bfV$ and return the orthonormal basis $\bfQ$. For the analysis of HOQRI, we are not required to understand retraction.

The formal definition of the tangent space is 
\[
\text{T}_{\bftheta}\tM = \{\tilde{\bftheta}'(0)~|~\tilde{\bftheta}:\bbR\rightarrow\tM \text{~is smooth at 0 and~}
\tilde{\bftheta}(0)=\bftheta\}.
\]
In other words, $\bfv$ is in $\text{T}_{\bftheta}\tM$ if and only if there exists a smooth curve $\tilde{\bftheta}(s)$ on $\tM$ passing through $\bftheta$ with velocity $\bfv$. This is important to define the directional derivative of a function $f$ on the manifold. For a point $\bftheta\in\tM$, the direction is limited to a vector $\bfv\in\text{T}_{\bftheta}\tM$, which is the velocity of a curve on the manifold at $\bftheta$, and the directional derivative is defined as
\begin{equation}\label{eq:dd}
f'(\bftheta;\bfv) = \lim_{s\downarrow0}\frac{f(\tilde{\bftheta}(s))-f(\bftheta)}{s},
\end{equation}
where $\tilde{\bftheta}:\bbR\rightarrow\tM$ satisfies $\tilde{\bftheta}(0)=\bftheta$ and $\tilde{\bftheta}'(0)=\bfv$. Although the choice of the curve $\tilde{\bftheta}(s)$ is not unique, it can be shown that they all lead to the same result. We are now ready to define the counterpart of gradient for functions on a manifold:
\begin{definition}[Riemannian gradient]\label{def:grad}
The Riemannian gradient of a smooth function $f$ at $\bftheta$, denoted as $\grad f(\bftheta)$, is a vector that satisfies
\[
\langle\grad f(\bftheta), \bfv\rangle = f'(\bftheta;\bfv),\quad
\forall \bftheta\in\tM, \bfv\in\text{T}_{\bftheta}(\tM).
\]
\end{definition}
Again, there is profound theory to establish its existence and that it is well-defined. A reassuring property is that this definition recovers the regular gradient for a function in the Euclidean space. In practice, the Riemannian gradient can be obtained as
\begin{equation}\label{eq:grad}
\grad f(\bftheta) = \Proj_{\bftheta}(\nabla f(\bftheta)),
\end{equation}
i.e., projecting the (unconstrained) gradient of $f$ onto the tangent space of $\tM$ at $\bftheta$. Specifically for the Tucker decomposition problem \eqref{prob:tucker} or in the manifold form \eqref{stiefelObjective}, if we denote the objective function as $f$, then
\[
\nabla_{\bfU^\mode{n}}f = 2\bfY_\mode{n}\bfG_\mode{n}^\T = 2\bfA^\mode{n},
\]
which is exactly the $\NewProd$ defined in \eqref{eq:ttmctc}. As a result, the corresponding Riemannian gradient is
\begin{align*}
\grad_{\bfU^\mode{n}}f &= \Proj_{\bfU^\mode{n}}(2\bfA^\mode{n}) \\
&= 2\bfA^\mode{n} - 2\bfU^\mode{n}
(\bfA^{\mode{n}\,\T}\bfU^\mode{n}+\bfU^{\mode{n}\,\T}\bfA^\mode{n})/2.
\end{align*}
Furthermore, we have that
\begin{equation}\label{eq:AU}
\bfA^{\mode{n}\,\T}\bfU^\mode{n}=\bfU^{\mode{n}\,\T}\bfA^\mode{n}=
\bfG_\mode{n}\bfG_\mode{n}^\T,
\end{equation}
as per the definition of $\bfA^\mode{n}$, so the final form of the Riemannian gradient is
\begin{equation}\label{eq:tucker_grad}
\grad_{\bfU^\mode{n}}f = \Proj_{\bfU^\mode{n}}(2\bfA^\mode{n}) = 
2\bfA^\mode{n} - 2\bfU^\mode{n}\bfG_\mode{n}\bfG_\mode{n}^\T.
\end{equation}

A necessary condition for optimality of the smooth manifold optimization problem \eqref{manifoldOptimization} is that $\grad f$ equals to zero. The intuition is that one should move a point on the manifold along curves, denoted as $\tilde{\bftheta}(s)$, which is tightly related to the Riemannian gradient; if a point is locally optimal, then all directional derivatives should be nonnegative, which then leads to the conclusion that the Riemannian gradient equals to zero. The notion of a stationary point follows naturally:
\begin{definition}[Stationary point]\label{def:stationary}
Given a smooth function $f$ on a Riemannian manifold $\tM$, we call $\bftheta\in\tM$ a stationary point of $f$ if $\grad f(\bftheta)=0$.
\end{definition}

It is of course equivalent to check whether the norm (squared) of the Riemannian gradient equals to zero. For the Tucker problem \eqref{prob:tucker}, if we denote the objective function as $f$, then we have that
\begin{align}\label{eq:grad2}
&\|\grad f(\{\bfU\})\|^2 = \sum_{n=1}^N \|2\bfA^\mode{n}-2\bfU^\mode{n}\bfG_\mode{n}\bfG_\mode{n}^\T\|^2 \nonumber \\
&= \sum_{n=1}^N 4\left(\|\bfA^\mode{n}\|^2 + \|\bfU^\mode{n}\bfG_\mode{n}\bfG_\mode{n}^\T\|^2 - 2\tr\bfA^{\mode{n}\,\T}\bfU^\mode{n}\bfG_\mode{n}\bfG_\mode{n}^\T \right) \nonumber \\
&= \sum_{n=1}^N 4\left(\|\bfA^\mode{n}\|^2 - \|\bfG_\mode{n}\bfG_\mode{n}^\T\|^2\right),
\end{align}
where the last equation is due to \eqref{eq:AU} and that $\bfU^\mode{n}$ is orthonormal.

\subsection{Convergence to a stationary point}\label{sec:conv}

We now show that HOQRI converges to a stationary point of \eqref{prob:tucker}, meaning that the Riemannian gradient of the objective \eqref{stiefelObjective}, with the product of $N$ Stiefel manifold, converges to zero. 

One of the major reasons why it is difficult to apply classical convergence results of BCD such as \cite{bertsekas2016nonlinear} in this case is that \cite{bertsekas2016nonlinear} requires each update to be the \emph{unique} optimizer of its corresponding subproblem, while the update rule in line 5 of Algorithm~\ref{alg:hoqri} is by nature not unique. However, what is essential is the subspace that the columns of $\bfU^\mode{n}$ represents, not the matrix per se. Therefore, we first define the ``distance'' between two $I\times K$ orthonormal matrices as
\begin{equation}\label{eq:dist}
d(\bfU,\bfV) = \min_{\bfQ^\T\bfQ=\bfQ\bfQ^\T=\bfI} \|\bfU-\bfV\bfQ\|^2.
\end{equation}
The minimization problem is the Procrustes projection problem, with solution given by the singular value decomposition of $\bfU^\T\bfV$ \cite{schonemann1966generalized}. This distance measures the difference between two subspaces, as it equal to zero if and only if $\bfU$ and $\bfV$ span the same subspace. One can also show that it is symmetric and it satisfies the triangular inequality, but we omit the proof here for conciseness. 

Similarly, we can define the distance between a set of orthonormal factor matrices $\{\bfU\}$ and a set $\cU=\cU^\mode{1}\times\cdots\times\cU^\mode{N}$ as
\[
d(\{\bfU\},\cU) = \sum_{n=1}^N 
\inf_{\bfV^\mode{n}\in\cU^\mode{n}}d(\bfU^\mode{n},\bfV^\mode{n}).
\]
The set of particular interest is the set of all stationary points of Problem \eqref{prob:main}, i.e., the set of all factor matrices such that their Riemannian gradients equal to zero:
\[
\cU^\ast = \{\{\bfU\}~|~ \grad f(\{\bfU\})=0\}.
\]

\begin{theorem}\label{thm:converge}
The iterates generated by HOQRI as in Alg.~\ref{alg:hoqri} satisfies
\begin{equation}\label{eq:converge}
\lim_{t\rightarrow\infty}~ d(\{\bfU_t\},\cU^\ast)=0.
\end{equation}
\end{theorem}
\begin{proof}
At iteration $t$ before the update, the objective value $f(\{\bfU_t\})$ is
\begin{align}\label{eq:ft}
f(\{\bfU_t\}) = 
\|\tX\times\{\bfU_t\}\|^2= \|\bfG_{\mode{n}\,t}\|^2.
\end{align}
The norm of its gradient with respect to $\bfU^\mode{n}$ is, according to \eqref{eq:grad2},
\begin{equation}\label{eq:grad_n}
\|\grad_{\bfU^\mode{n}}f(\{\bfU_t\})\|^2 = 4\|\bfA^\mode{n}_t\|^2 - 
4\|\bfG_{\mode{n}\,t}\bfG_{\mode{n}\,t}^\T\|^2.
\end{equation}
Denoting the singular values of $\bfA^\mode{n}_t$, in descending order, as $\sigma_1\geq\cdots\geq\sigma_{K_n}$ and the eigenvalues of $\bfG_{\mode{n}\,t}\bfG_{\mode{n}\,t}^\T$ as $\lambda_1\geq\cdots\geq\lambda_{K_n}$, we can rewrite \eqref{eq:grad_n} as
\begin{equation}\label{eq:grad_nsl}
\|\grad_{\bfU^\mode{n}}f(\{\bfU_t\})\|^2 = 4\sum_{k=1}^{K_n}(\sigma_k^2-\lambda_k^2).
\end{equation}

After the update, according to the derivation in the proof of Theorem~\ref{thm:obj}, $\bfU^\mode{n}_{t+1}$ yields the same objective as the orthonormal matrix obtained from the Procrustus rotation, which maximizes a linear lowerbound function, therefore we have
\begin{align}\label{eq:per-iter0}
f(\{\bfU_{t+1}\})
&= f(\{\bfU^\mode{1}_t,\ldots,\bfU^\mode{n-1}_t,\bfQ_t\bfV_t^\T, \bfU^\mode{n+1}_t,\ldots,\bfU^\mode{N}_t\}) \nonumber \\
&\geq f(\{\bfU_{t}\}) + 
2\tr\bfA_t^{\mode{n}\,\T}(\bfQ_t\bfV_t^\T-\bfU^\mode{n}_t) \nonumber \\
&= \|\bfG_{\mode{n}\,t}\|^2 + 2\tr\bfA^{\mode{n}\,\T}_t\bfQ_t\bfV_t^\T - 2\|\bfG_{\mode{n}\,t}\|^2 \nonumber \\ 
&= 2\|\bfA^\mode{n}_t\|_\ast - \|\bfG_{\mode{n}\,t}\|^2,
\end{align}
where the third line comes from $\tr\bfA^{\mode{n}\,\T}_t\bfU^\mode{n}_{t}
=\tr\bfG_{\mode{n}\,t}\bfG_{\mode{n}\,t}^\T=\|\bfG_{\mode{n}\,t}\|^2$, and the fourth line is because $\tr\bfA^{\mode{n}\,\T}_t\bfU$ is maximized over the Stiefel manifold when $\bfU=\bfQ_t\bfV_t^\T$ where $\bfQ_t$ and $\bfV_t$ collect the left and right singular vectors of $\bfA_t^\mode{n}$ as their columns, respectively, at which point $\tr\bfA^{\mode{n}\,\T}_t\bfQ\bfV^\T$ equals to the sum of the singular values of $\bfA^{\mode{n}}_t$, i.e., its nuclear norm $\|\bfA^\mode{n}_t\|_\ast$. Combining \eqref{eq:per-iter0} with \eqref{eq:ft}, we have
\begin{equation}\label{eq:A+G<f-f}
2\|\bfA^\mode{n}_t\|_\ast - 2\tr\bfG_{\mode{n}\,t}\bfG_{\mode{n}\,t}^\T \leq 
f(\{\bfU_{t+1}\})-f(\{\bfU_{t}\}).
\end{equation}
Recall our definition of the singular values of $\bfA^\mode{n}_t$ and the eigenvalues of $\bfG_{\mode{n}\,t}\bfG_{\mode{n}\,t}^\T$, we can also rewrite \eqref{eq:A+G<f-f} as
\begin{equation}\label{eq:s+l<f-f}
2\sum_{k=1}^{K_n}(\sigma_k-\lambda_k)\leq 
f(\{\bfU_{t+1}\})-f(\{\bfU_{t}\}).
\end{equation}
We will next try to upperbound \eqref{eq:grad_nsl} by \eqref{eq:s+l<f-f}.

As shown in Lemma~\ref{lmm:interlacing} in the Appendix using the Cauchy interlacing theorem~\cite[pp.118]{bellman1997introduction}, we have that
\begin{equation}\label{eq:interlacing}
\sigma_k\geq\lambda_k, k=1,\ldots,K_n.
\end{equation}
This shows that all $K_n$ terms in \eqref{eq:grad_nsl} are nonnegative, and we can safely bound it as 
\begin{align}\label{eq:bound-grad1}
\|\grad_{\bfU^\mode{n}}f(\{\bfU_t\})\|^2 
&= 4\sum_{k=1}^{K_n}(\sigma_k-\lambda_k)(\sigma_k+\lambda_k) \nonumber\\
&\leq 4\sum_{k=1}^{K_n}(\sigma_k-\lambda_k)\sum_{k=1}^{K_n}(\sigma_k+\lambda_k).
\end{align}
Equation \eqref{eq:interlacing} also necessarily means that $\|\bfA^\mode{n}_t\|_\ast\geq\|\bfG_{\mode{n}\,t}\|^2$; combining it with \eqref{eq:per-iter0} also shows that $f(\{\bfU_{t+1}\}) \geq\|\bfA^\mode{n}_t\|_\ast$. In other words, we have
\begin{align*}
f(\{\bfU_{t}\}) = \sum_{k=1}^{K_n}\lambda_k \quad\text{and}\quad
f(\{\bfU_{t+1}\}) \geq \sum_{k=1}^{K_n}\sigma_k.
\end{align*}
Furthermore, let $f_{\sup}$ denote the theoretical maximum of problem~\eqref{prob:main}, then obviously we have
\[
\sum_{k=1}^{K_n}\lambda_k \leq \sum_{k=1}^{K_n}\sigma_k \leq f_{\sup}.
\]
Continuing \eqref{eq:bound-grad1} gives us
\begin{equation}\label{eq:bound-grad}
\|\grad_{\bfU^\mode{n}}f(\{\bfU_t\})\|^2 \leq
8f_{\sup}\sum_{k=1}^{K_n}(\sigma_k-\lambda_k).
\end{equation}
Combining \eqref{eq:bound-grad} with \eqref{eq:s+l<f-f} shows
\begin{equation}\label{eq:per-iter}
\|\grad_{\bfU^\mode{n}}f(\{\bfU_t\})\|^2 \leq 
4f_{\sup}\left( f(\{\bfU_{t+1}\})-f(\{\bfU_{t}\}) \right),
\end{equation}
where $n=(t\mod N) + 1$.

Inequality \eqref{eq:per-iter} is an important inequality as it establishes the relationship between the objective improvement and the Riemannian gradient of the block of variables that is being updated. Summing up \eqref{eq:per-iter} over $t=0,\ldots,T$ yields
\begin{align*}
\sum_{t=0}^T\|\grad_{\bfU^\mode{n}}f(\{\bfU_{t}\})\|^2
&\leq \sum_{t=0}^T4f_{\sup}(f(\{\bfU_{t+1}\})\!-\!f(\{\bfU_{t}\}))\\
&= 4f_{\sup}(f(\{\bfU_{T+1}\}) - f(\{\bfU_{0}\})) \\
&\leq 4f_{\sup}^2.
\end{align*}
Letting $T$ go to infinity, we get an infinite sum of nonnegative numbers on the left-hand-side, while the right-hand-side is a finite constant that does not change with $T$. This shows that the partial Riemannian gradients converges to zero, i.e., 
\begin{equation}\label{eq:partial_grad_0}
\grad_{\bfU^\mode{n}}f(\{\bfU_{t}\}) \rightarrow0,
\end{equation}
again with $n=(t\mod N)+1$.

Consider a limit point $\{\bar{\bfU}\}$ and a subsequence $\{\bfU_{t_j}\}_j$ converging to $\{\bar{\bfU}\}$. Here we are using the distance defined in \eqref{eq:dist}, meaning that for all $\epsilon>0$, there exists an index $T$ such that
\[
d(\bfU^\mode{n}_{t_j},\bar{\bfU}^\mode{n})<\epsilon, \quad
\forall t_j>T \text{ and }
n=1,\ldots,N.
\]
Since there is a finite number of $N$ blocks, there exists a block that is updated infinitely often in the subsequence $\{\bfU_{t_j}\}_j$. Assume without loss of generality that the first block is updated infinitely often, then we directly have that
\begin{equation}\label{eq:grad1_converge}
\grad_{\bfU^\mode{1}}f(\{\bar{\bfU}\})=0.
\end{equation}
What we aim to show next is that the partial gradients with respect to other blocks at $\{\bar{\bfU}\}$ are also zero, even if they are not updated infinitely often in the subsequence $\{\bfU_{t_j}\}_j$.

Due to \eqref{eq:partial_grad_0}, and recall the definition of $\grad f$ \eqref{eq:grad}, we rewrite it as
\begin{equation}\label{eq:grad_norm_0}
\|2\bfA^\mode{n}_t-2\bfU^\mode{n}_t\bfG_{\mode{n}\,t}\bfG_{\mode{n}\,t}^\T\|^2
\rightarrow 0.
\end{equation}
This implies that as $t\rightarrow\infty$, $\bfA^\mode{n}_t$ and $\bfU^\mode{n}_t$ span the same subspace in the limit. Recall that $\bfU^\mode{n}_{t+1}$ is defined as an orthonormal basis of $\bfA^\mode{n}_t$, therefore
\begin{equation}\label{eq:per-iter-0}
d(\bfU^\mode{n}_{t+1},\bfU^\mode{n}_{t})\rightarrow0.
\end{equation}
This, together with the subsequence $\{\bfU_{t_j}\}_j$ that converges to $\{\bar{\bfU}\}$, implies that the subsequence $\{\bfU_{t_j+1}\}_j$ also converges to $\{\bar{\bfU}\}$. Since in subsequence $\{\bfU_{t_j}\}_j$ the first block is updated infinitely often, in subsequence $\{\bfU_{t_j+1}\}_j$ the second block is updated infinitely often, therefore
\[
\grad_{\bfU^\mode{2}}f(\{\bar{\bfU}\})=0.
\]
Repeat the steps for the other factors, we conclude that
\[
\grad f(\{\bar{\bfU}\})=0.
\]
In other words, every limit point is a stationary point.

Finally, we show that the whole sequence $\{\bfU_t\}$ satisfies \eqref{eq:converge} by contradiction. Suppose on the contrary that there exists a subsequence such that $d(\{\bfU_{t_j}\},\cU^\ast)>\gamma$ for some $\gamma>0$. Since the Stiefel manifold is a compact set, it has a limit point $\{\bar{\bfU}\}$, which implies $d(\{\bar{\bfU}\},\cU^\ast)>\gamma$. This contradicts the fact that $\{\bar{\bfU}\}\in\cU^\ast$ as per our previous argument. This completes the proof that the iterates of HOQRI, as shown in Alg.~\ref{alg:hoqri}, converges to the set of limit points of \eqref{prob:main}.
\end{proof}

\noindent
\textbf{Remark.}
For those who are familiar with the convergence of block coordinate descent (BCD) \cite{bertsekas2016nonlinear} and its variants such as block successive upperbound minimization (BSUM) \cite{razaviyayn2013unified}, Theorem~\ref{thm:converge} may come as somewhat unexpected as it does not rely on the assumption that each update is the \emph{unique} optimizer of a certain subproblem. In fact, by nature the update cannot be unique since any orthonormal basis of $\bfA^\mode{n}$ can be used as the update for $\bfU^\mode{n}$. What makes it different is that thanks to the specific problem structure, the norm of the (partial) Riemannian gradient at every iteration can be upperbounded by the per-iteration objective improvement. This is not guaranteed by a generic BCD or BSUM algorithm. In fact, the proof is more similar to that of the block successive convex approximation (BSCA) algorithm as in Theorem 4 of \cite{razaviyayn2013unified}. The statement of Theorem~\ref{thm:converge} also implies that the whole sequence $\{\bfU_t\}$ does not necessarily converge---the claim is that limit points exist and all of them are stationary points. 

One may wonder if this new manifold optimization perspective on the Tucker decomposition problem may improve the convergence analysis of the well-known HOOI. In our attempt, all the steps in the proof of Theorem~\ref{thm:converge} can be repeated until \eqref{eq:grad1_converge}, except the first line of \eqref{eq:per-iter0} becomes $\geq$, which does not affect the proof: this is because if $\bfU^\mode{n}_{t+1}$ is obtained from HOOI as in line 5 of Algorithm~\ref{alg:hooi}, by definition it yields larger objective value than that of $\bfQ_t\bfV_t^\T$. However, we cannot guarantee \eqref{eq:per-iter-0} using \eqref{eq:grad_norm_0} if $\bfU_{t+1}^\mode{n}$ is obtained from HOOI: \eqref{eq:grad_norm_0} implies that in the limit, $\bfU_t^\mode{n}$ spans the same subspace of \emph{some} $K_n$ eigenvectors of $\bfY_{\mode{n}\,t}\bfY_{\mode{n}\,t}^\T$, but $\bfU_{t+1}^\mode{n}$ must be the eigenvectors of the $K_n$ largest eigenvalues, therefore it is not obvious that $d(\bfU^\mode{n}_{t+1},\bfU^\mode{n}_{t})\rightarrow0$ for HOOI. The best known convergence for HOOI is still by Xu et al. \cite{xu2018convergence}, which requires a so-called ``non-degeneracy'' assumption to guarantee convergence to a stationary point as we explained in Section \ref{sec:tucker}. This assumption is not required by HOQRI to guarantee convergence.

\section{Numerical Experiments}
We now demonstrate the performance of the proposed HOQRI algorithm for Tucker decomposition. The HOQRI algorithm has been implemented in both Matlab and C++. 

In both implementations, we employ a sparse tensor format for data storage, ensuring that the space utilization is directly proportional to the number of non-zero entries rather than the tensor's dimensions. In the Matlab implementation, we utilize the same sparse tensor format employed by Tensor Toolbox \cite{tensor_toolbox} and Tensorlab \cite{tensorlab3.0}. Using this tensor format in our implementation of HOQRI allows us to fairly assess its performance compared to the Tensor Toolbox and Tensorlab methods for low multi-linear rank tensor approximation. In the C++ implementation, we adopt the data structure utilized by S-HOT \cite{oh2017s}.

Reproducibility: The source code of HOQRI for the MATLAB results displayed in Section~\ref{sec:matlab} are available at \url{https://github.com/amitbhat31/hoqri}. The source code for the C++ implementation of HOQRI and the datasets examined in Section~\ref{sec:cpp} are available at \url{https://github.com/yuchen19253/HOQRI}.

\subsection{Matlab implementation}\label{sec:matlab}

We perform a comparative analysis of HOQRI implemented in Matlab with the algorithms available in the Tensor Toolbox \cite{tensor_toolbox} and Tensorlab \cite{tensorlab3.0} packages. In Tensor Toolbox, we consider the \texttt{tucker\_als} methods, which implements the HOOI algorithm and is the only method that utilizes Tensor Toolbox's sparse tensor structure. Meanwhile, Tensorlab provides a variety of algorithms for Tucker decomposition. In addition to the HOOI implementation provided in \texttt{lmlra\_hooi}, Tensorlab also features \texttt{lmlra\_minf} and \texttt{lmlra\_nls}, which employ nonlinear unconstrained optimization and nonlinear least squares, respectively. 

In our implementation of HOQRI, we utilize the sparse tensor format provided by the Tensor Toolbox and Tensorlab packages. However, we refrain from using any optimized methods from the libraries for computing sparse mode products, such as the \texttt{ttv} function in Tensor Toolbox and the \texttt{mtkronprod} function in Tensorlab. This demonstrates that HOQRI yields significant convergence and speed improvements without the need for any extra scaffolding other than the sparse tensor structure as opposed to the baseline methods, which use the aforementioned optimized mode products mentioned above.

Since the methods \texttt{lmlra\_minf} and \texttt{lmlra\_nls} do not impose orthonormal constraints on the set of matrices ${\bfU}$, it is not appropriate to evaluate the objective \eqref{prob:main}. Consequently, we consider the standard reconstruction objective~\eqref{prob:tucker}. We evaluate the scalability of HOQRI by considering how the per-iteration complexity changes as tensor dimensionality and the rank of the core tensor increases.

\begin{figure}[t]
\centering
\begin{subfigure}[t]{0.49\linewidth}  
    \centering 
    \includegraphics[width=\linewidth]{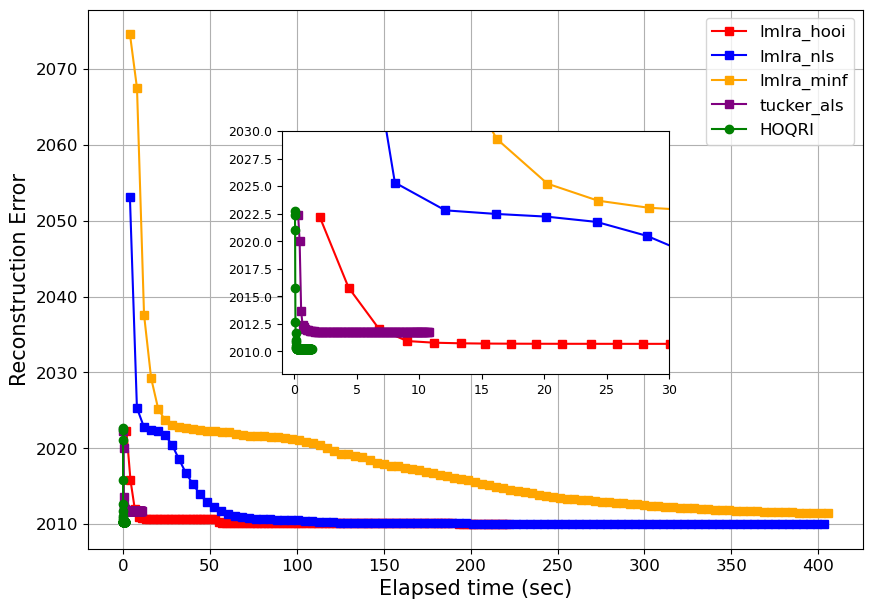}
    \caption[]%
    {nnz=$10I$.}    
\end{subfigure}
\hfill
\begin{subfigure}[t]{0.49\linewidth}
    \centering
    \includegraphics[width=\linewidth]{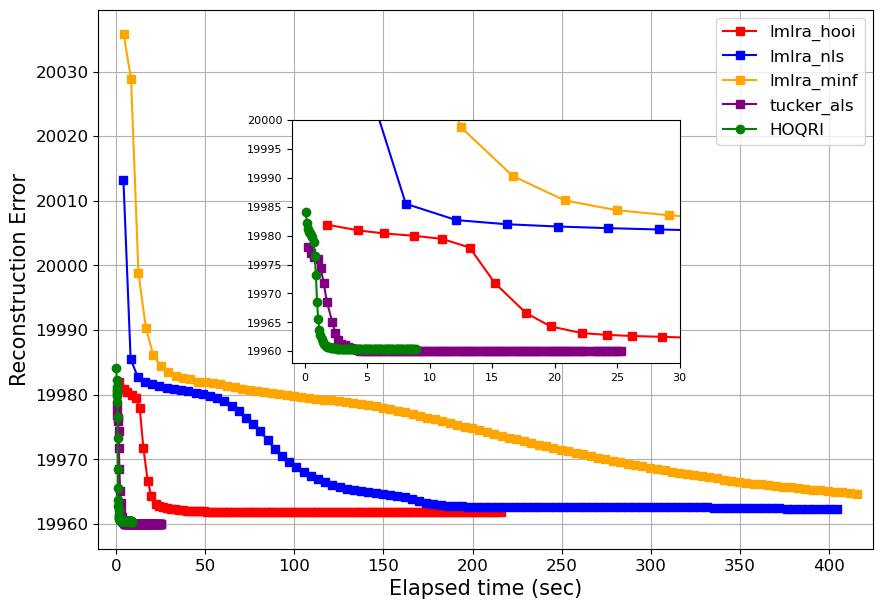}
    \caption[]%
    {nnz=$100I$.}    
\end{subfigure}
\caption{Convergence of HOQRI on 3-dimensional tensors versus the methods lmlra\_hooi, lmlra\_minf, lmlra\_nls, and tucker\_als with $I=600$.} 
\label{fig:3mode_error}
\end{figure}

\subsubsection{Convergence over time}
To test the convergence of the algorithms, we synthetically generate sparse tensors of size $\underbrace{I \times I \times \cdots \times I}_{n}$ for $n = 3, 4, 5$ where $I$ denotes the size of the tensor dimension and $n$ denotes the number of tensor dimensions. The target multilinear rank in this context is set to $\underbrace{10 \times 10 \times \cdots \times 10}_{n}$. We consider tensors with number of nonzero entries $10I$ and $100I$. For tensors of $n=4$ and $n=5$ dimensions, we only consider HQORI compared to \texttt{lmlra\_hooi} and \texttt{tucker\_als}, as the two Tensorlab methods without orthonormal constraints on the factor matrices fail to scale for more than 3 dimensions and take exceptionally long.

To take into account all factors that impact efficiency, we display time in seconds on the horizontal axis and approximation error \eqref{prob:tucker} on the vertical axis. All methods are run for a maximum of 100 iterations or until the change in the approximation error is less than $10^{-12}$. HOQRI achieves superior performance in both approximation error and computational speed across all sizes of tensors, as demonstrated in Figures 
\ref{fig:3mode_error}, \ref{fig:4mode_error}, and \ref{fig:5mode_error}.

\begin{figure}[t]
\centering
\begin{subfigure}[t]{0.49\linewidth}  
    \centering 
    \includegraphics[width=\linewidth]{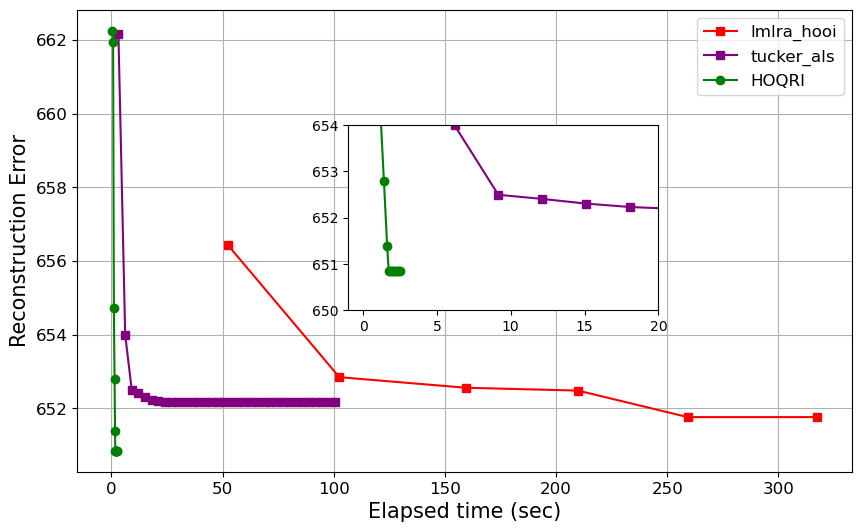}
    \caption[]%
    {nnz=$10I$.}    
\end{subfigure}
\hfill
\begin{subfigure}[t]{0.49\linewidth}
    \centering
    \includegraphics[width=\linewidth]{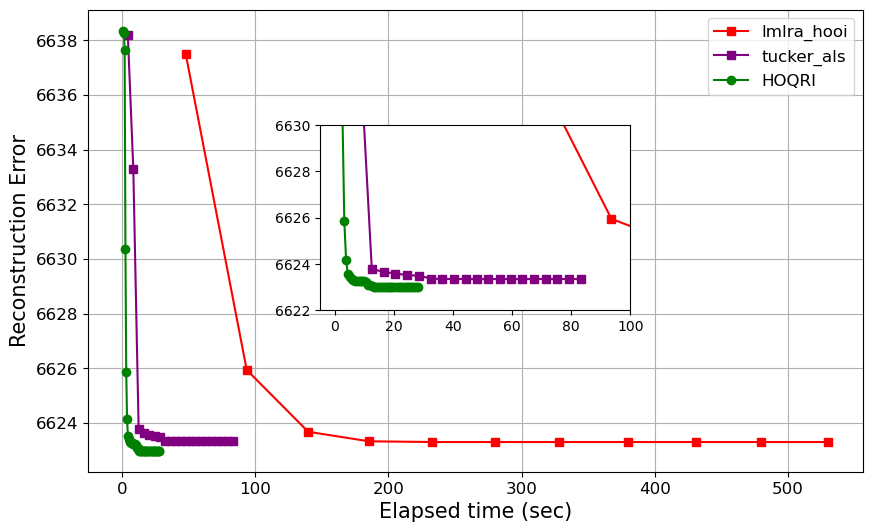}
    \caption[]%
    {nnz=$100I$.}    
\end{subfigure}
\caption{Convergence of HOQRI on 4-dimensional tensors versus the methods lmlra\_hooi and tucker\_als with $I=200$.} 
\label{fig:4mode_error}
\end{figure}

\begin{figure}[t]
\centering
\begin{subfigure}[t]{0.49\linewidth}  
    \centering 
    \includegraphics[width=\linewidth]{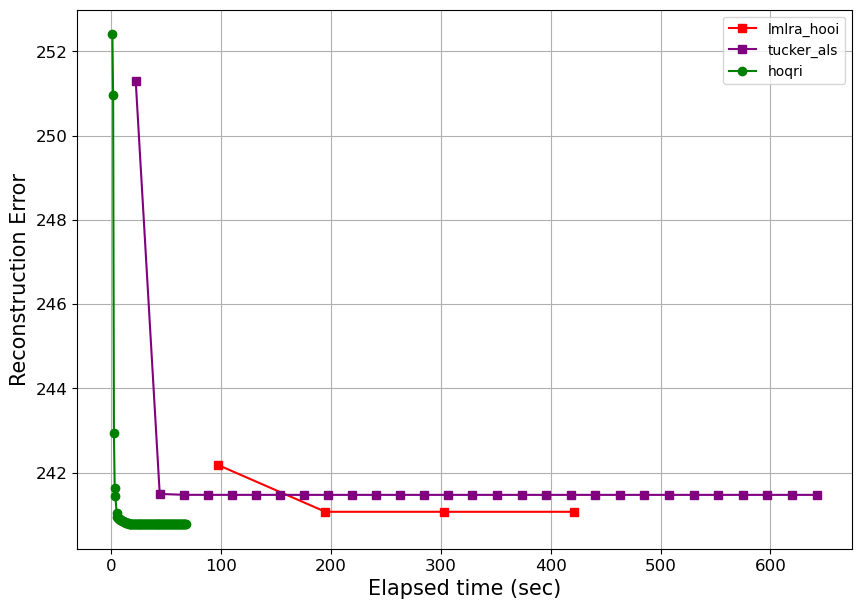}
    \caption[]%
    {nnz=$10I$.}    
\end{subfigure}
\hfill
\begin{subfigure}[t]{0.49\linewidth}
    \centering
    \includegraphics[width=\linewidth]{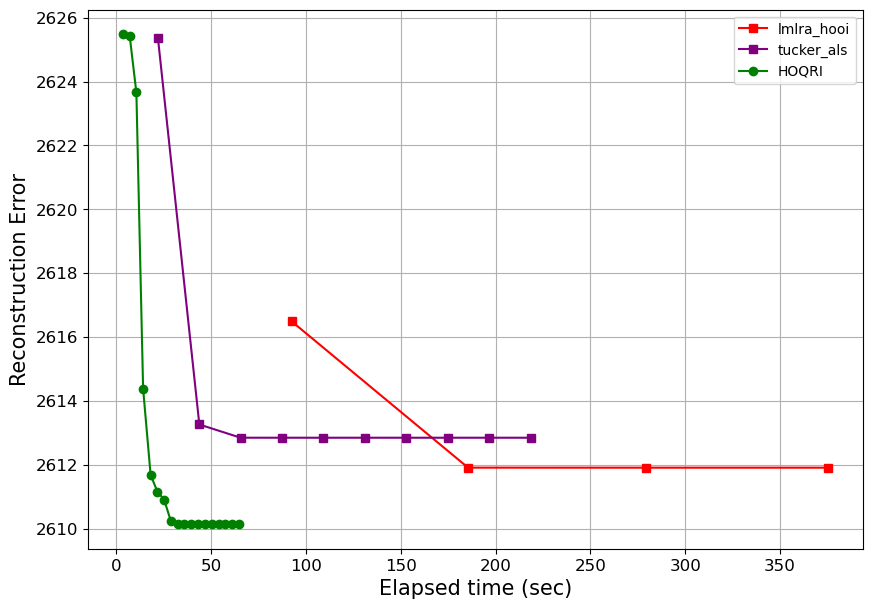}
    \caption[]%
    {nnz=$100I$.}    
\end{subfigure}
\caption{Convergence of HOQRI on 5-dimensional tensors versus the methods lmlra\_hooi and tucker\_als with $I=80$.} 
\label{fig:5mode_error}
\end{figure}

\subsubsection{Per-iteration complexity}
We now take a more detailed look at the efficiency of HOQRI compared to Tensor Toolbox and Tensorlab algorithms with respect to per-iteration complexity. First, we consider the averaged elapsed time per iteration as we increase the tensor dimension $I$ from $10^2$ to $3 \times 10^3$ with $\nnz = 10I$ and $\nnz=100I$ for 3-way tensors $I \times I \times I$. The core rank is fixed as $10 \times 10 \times 10$. All experiments were run for 100 iterations. As shown in Figure~\ref{fig:tensordim},
our algorithm is more scalable than all the baselines. Tensorlab methods eventually run out of memory, with \texttt{lmlra\_nls} and \texttt{lmlra\_minf} failing at $I = 700$ and \texttt{lmlra\_hooi} failing at $I = 1000$. This is due to the Tensorlab methods not supporting sparse tensors as inputs. In contrast, \texttt{tucker\_als} works with sparse tensors, and demonstrates comparable per-iteration complexity for $I \leq 1000$. However, as I increases, \texttt{tucker\_als} takes longer for each iteration. Meanwhile, the average elapsed time per iteration of HOQRI does not change much as the tensor size increases.

Next, we fix the tensor dimension to be $I = 600$ and evaluate the scalability of HOQRI as the core rank increases from 4 to 20. All experiments were run for 100 iterations. Figure~\ref{fig:corerank}
demonstrates that HOQRI consistently runs faster and scales at a far better rate compared to the Tensorlab methods. \texttt{tucker\_als} is once again the only algorithm that demonstrates similar scalability to HOQRI as the core rank increases, but HQORI consistently still has a lower per-iteration speed.

\begin{figure}[t]
\centering
\begin{subfigure}[t]{0.49\linewidth}  
    \centering 
    \includegraphics[width=\linewidth]{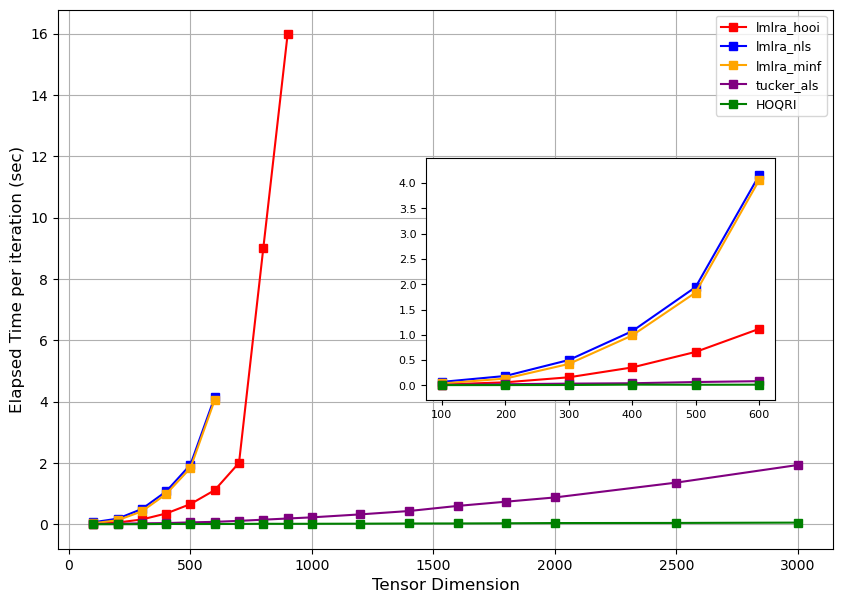}
    \caption[]%
    {nnz=$10I$.}    
\end{subfigure}
\hfill
\begin{subfigure}[t]{0.49\linewidth}
    \centering
    \includegraphics[width=\linewidth]{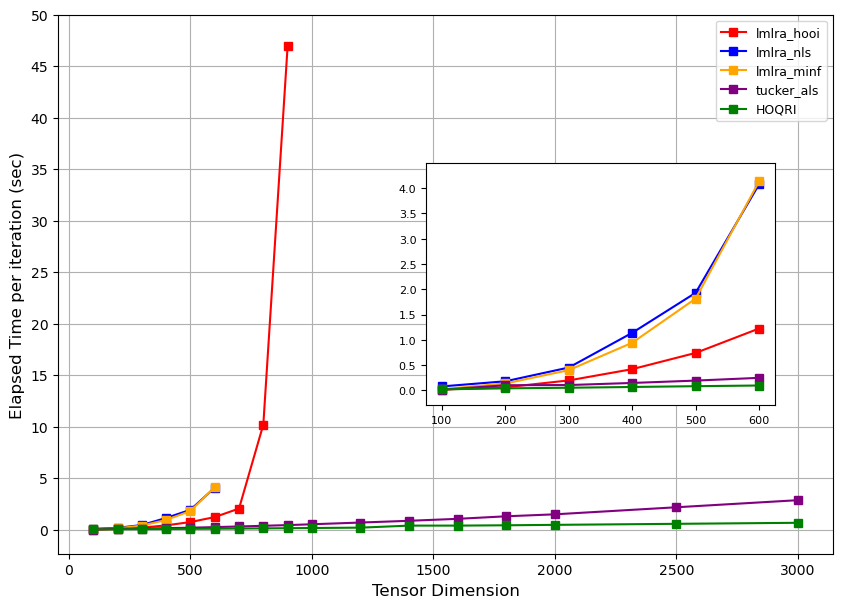}
    \caption[]%
    {nnz=$100I$.}    
\end{subfigure}
\caption{Elapsed time per iteration as tensor dimension increases for 3-dimensional tensors with core rank fixed at $10 \times 10 \times 10$.} 
\label{fig:tensordim}
\end{figure}

\begin{figure}[t]
\centering
\begin{subfigure}[t]{0.49\linewidth}  
    \centering 
    \includegraphics[width=\linewidth]{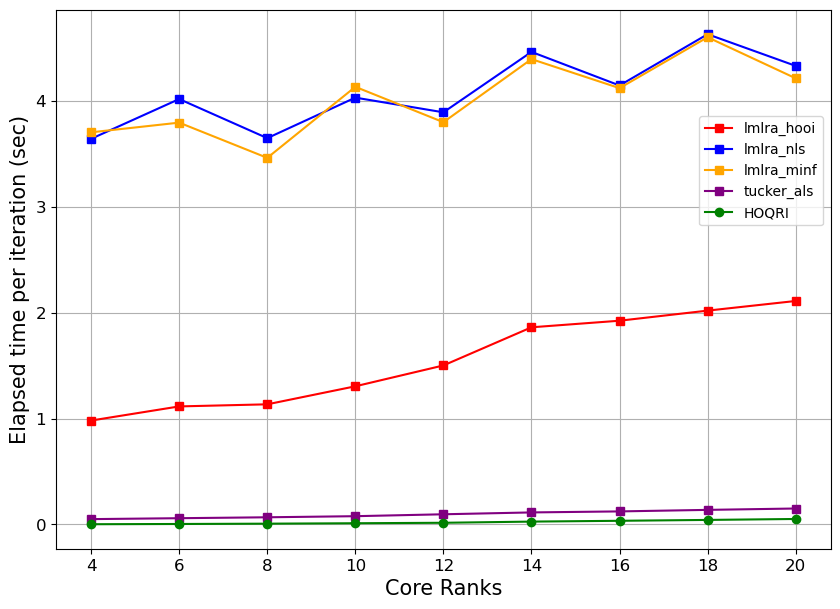}
    \caption[]%
    {nnz=$10I$.}    
\end{subfigure}
\hfill
\begin{subfigure}[t]{0.49\linewidth}
    \centering
    \includegraphics[width=\linewidth]{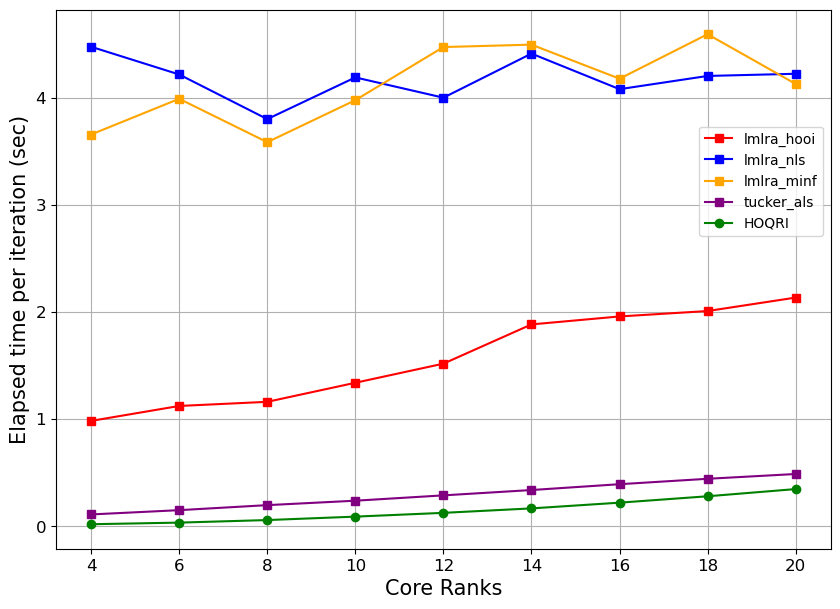}
    \caption[]%
    {nnz=$100I$.}    
\end{subfigure}
\caption{Elapsed time per iteration as core rank increases for 3-dimensional tensors with tensor dimension fixed at $600 \times 600 \times 600$.} 
\label{fig:corerank}
\end{figure}

\subsection{C++ implementation}\label{sec:cpp}
In this subsection, we conduct several experiments to compare our model with Scalable High-Order Tucker Decomposition (S-HOT) by Jinoh et al. \cite{oh2017s}, which is developed in the C++ environment. In this case, we evaluate the performances of HOQRI vis-a-vis S-HOT on the following real-world data sets, summarized in Table \ref{datasets}:
\begin{itemize}
\item The Delicious and Last data sets are time-evolving data\footnote{https://grouplens.org/datasets/hetrec-2011/}.
The Delicious data set \cite{Delicious} contains social networking, bookmarking, and tagging information from a set of 2K users from the Delicious social bookmarking system \footnote{http://www.delicious.com}. The Last data set \cite{Last} contains social networking, tagging, and music artist listening information from a set of 2K users from an online music system \footnote{http://www.lastfm.com}.
\item The Facebook data set is the target network for a small subset of posts on other users' walls on Facebook \cite{konect:2014:facebook-wosn-wall}\cite{viswanath09}\cite{konect}. The nodes of the network are Facebook users. The first and second columns in the dataset represent the user ID. The third dimension is a timestamp. The value of each entry represents the weight of the edge in the network.
\item MovieLens and 10M-MovieLens are movie rating data\footnote{https://grouplens.org/datasets/movielens/} consisting of (user, movie, time, rating) \cite{MovieLen}. MovieLens is a smaller subset with a higher density, while 10M-MovieLens is bigger.
\end{itemize}

\begin{table}[H]
\centering
\caption{Summary of real-world tensor datasets used for comparison.}\label{datasets}
\begin{tabular}{cccccc}
   \toprule
   Dataset & Dimension & $\nnz$\\
   \midrule
   Last & (2100, 18744, 12647) & 186479 \\
   Delicious & (108035, 107253, 52955) & 437593 \\
   Facebook & (46952, 46951, 1592) & 738079 \\
   MovieLens & (610, 49961, 8215) & 84159\\
   10M-MovieLens & (162541, 49994, 9083) & 20503478\\
   \bottomrule
\end{tabular}
\end{table}

\begin{figure}[H]
\centering
\begin{subfigure}[t]{0.49\linewidth}  
    \centering 
    \includegraphics[width=\linewidth]{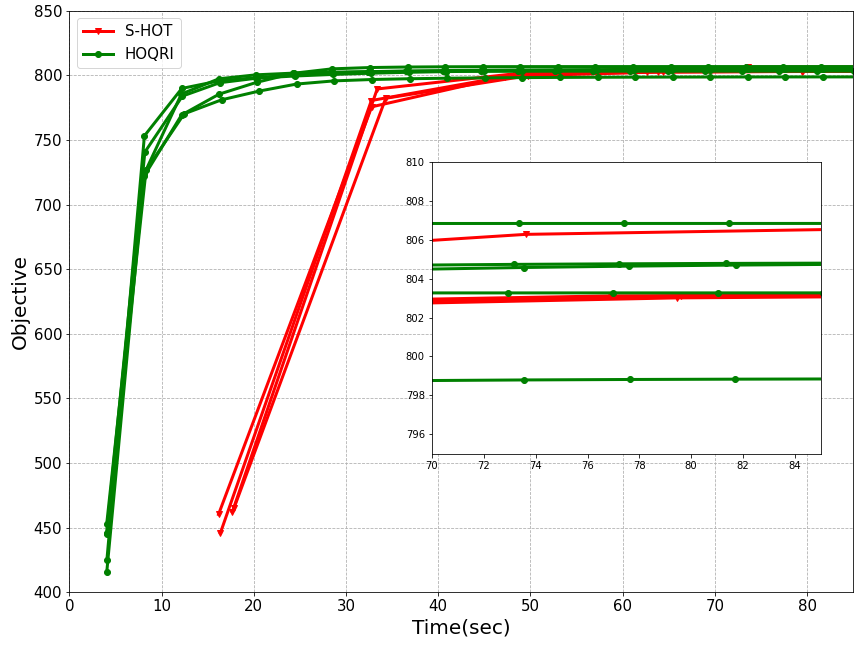}
    \caption{Rank $10\times10\times10$.} 
    \label{Last}
\end{subfigure}
\hfill
\begin{subfigure}[t]{0.49\linewidth}
    \centering
    \includegraphics[width=\linewidth]{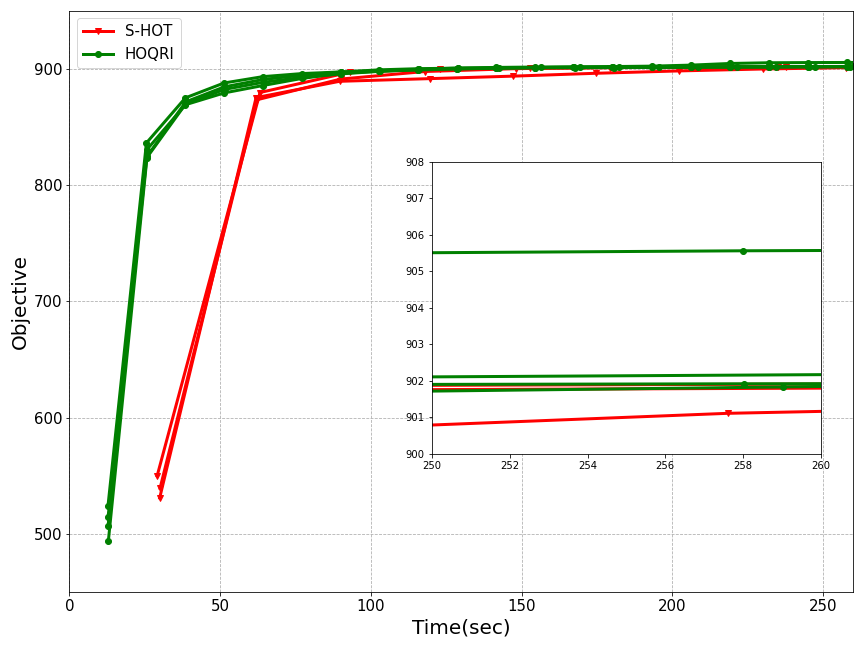}
    \caption{Rank $15\times15\times15$.} 
    \label{Last151515}
\end{subfigure}
\caption{Convergence on the Last data set.} 
\end{figure}

\begin{figure}[H]
\begin{subfigure}[t]{0.49\linewidth}  
    \centering 
    \includegraphics[width=\linewidth]{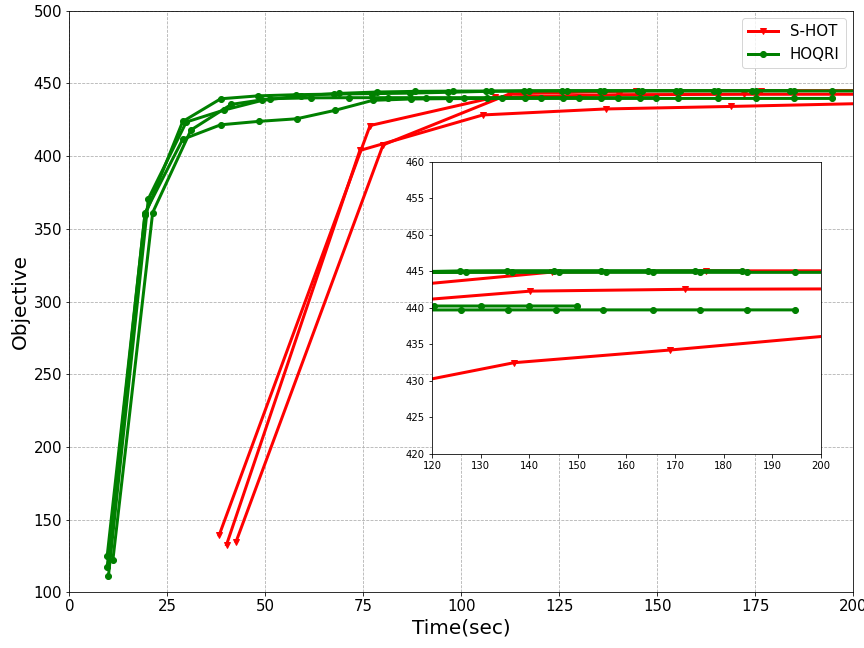}
    \caption{Rank $10\times10\times10$.} 
    \label{Delicious}
\end{subfigure}
\hfill
\begin{subfigure}[t]{0.49\linewidth}
    \centering
    \includegraphics[width=\linewidth]{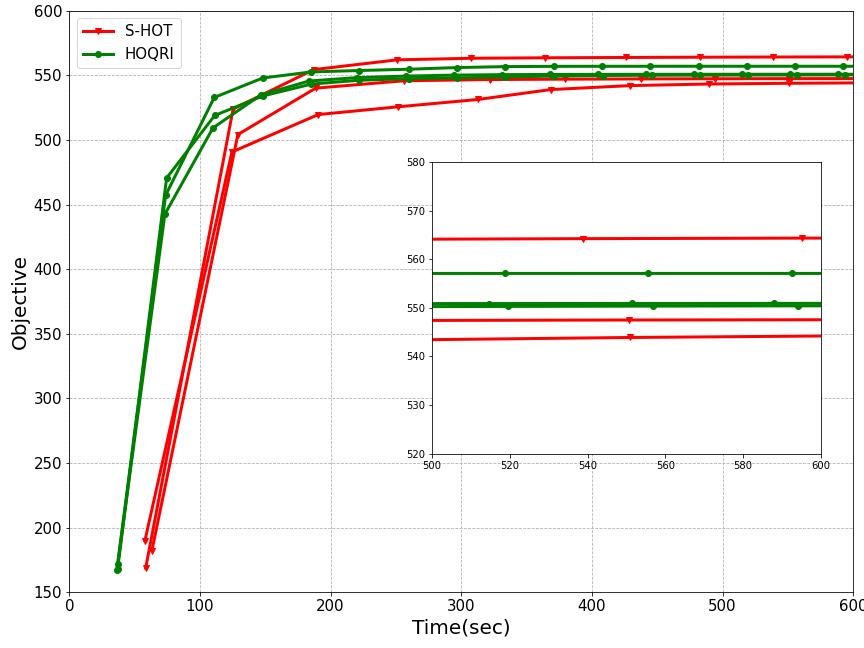}
    \caption{Rank $20\times20\times10$.} 
    \label{Delicious202010}
\end{subfigure}
\caption{Convergence on the Delicious data set.} 
\end{figure}

\begin{figure}[H]

\begin{subfigure}[t]{0.49\linewidth}  
    \centering 
    \includegraphics[width=\linewidth]{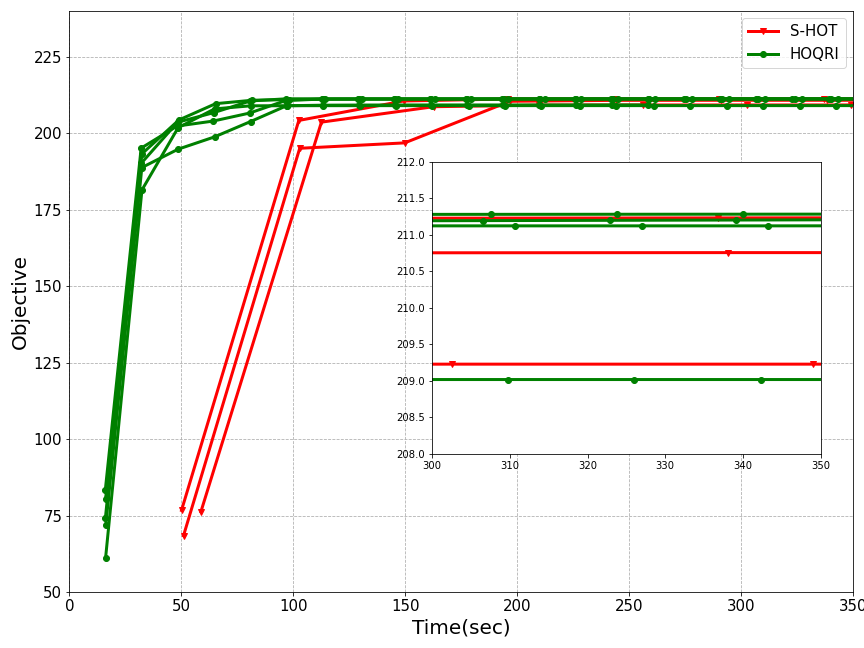}
    \caption{Rank $10\times10\times10$.} 
    \label{Facebook}
\end{subfigure}
\hfill
\begin{subfigure}[t]{0.49\linewidth}
    \centering
    \includegraphics[width=\linewidth]{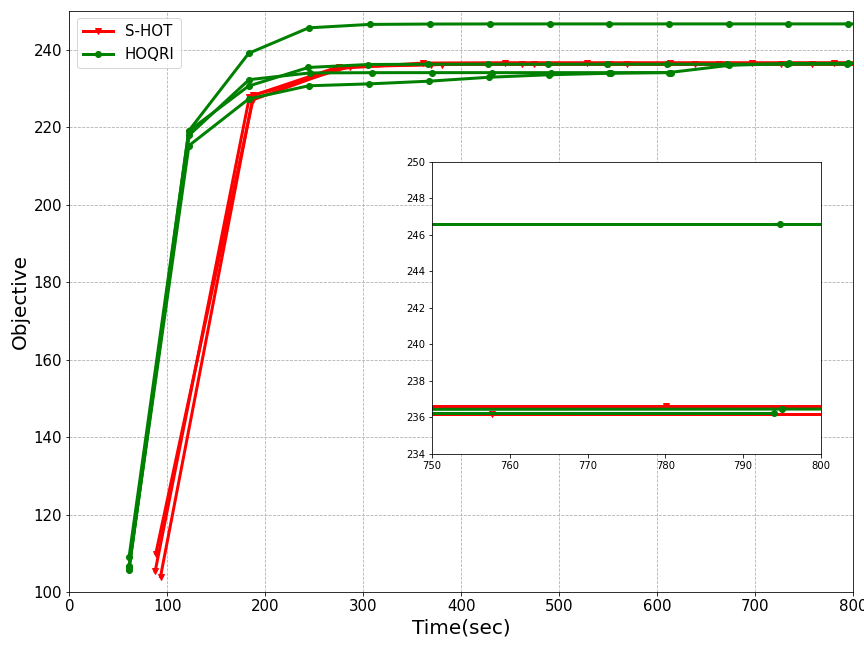}
    \caption{Rank $20\times20\times10$.} 
    \label{Facebook202010}
\end{subfigure}
\caption{Convergence on the Facebook data set.} 
\end{figure}

\noindent The convergence of HOQRI versus S-HOT on all these data sets are shown in Figures \ref{Last}, Figure \ref{Delicious}, Figure \ref{Facebook}, Figure \ref{MovieLen} and Figure \ref{10M_MovieLen}. It is evident that HOQRI converges significantly faster than S-HOT in all cases.

To make the task more challenging, we try relatively higher multilinear ranks to test their performance, as described in Figures \ref{Last151515}, Figure \ref{Delicious202010}, Figure \ref{Facebook202010}, Figure \ref{MovieLen152515} and Figure \ref{10M_MovieLen202010} on the right of the figures. The compact subfigure enclosed within the larger illustration delves into the convergence aspect of both algorithms. These subfigures clearly indicate that HOQRI exhibits the capability to converge towards a superior solution for the maximization of the objective \eqref{prob:main}. We tested several times to compare HOQRI and S-HOT to show that the result was not from some lucky initialization. For the larger dataset 10M-MovieLens, it is only run once as S-HOT takes 2500 seconds to 3000 seconds per iteration, which would require two days to finish 50 iterations.

\begin{figure}[H]
\begin{subfigure}[t]{0.49\linewidth}  
    \centering 
    \includegraphics[width=\linewidth]{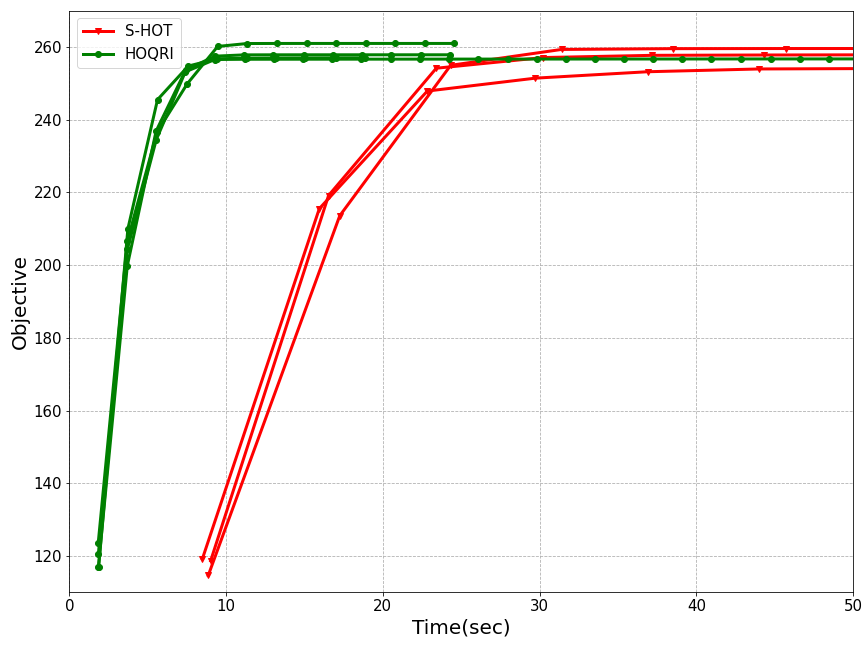}
    \caption{Rank $10\times10\times10$.} 
    \label{MovieLen}
\end{subfigure}
\hfill
\begin{subfigure}[t]{0.49\linewidth}
    \centering
    \includegraphics[width=\linewidth]{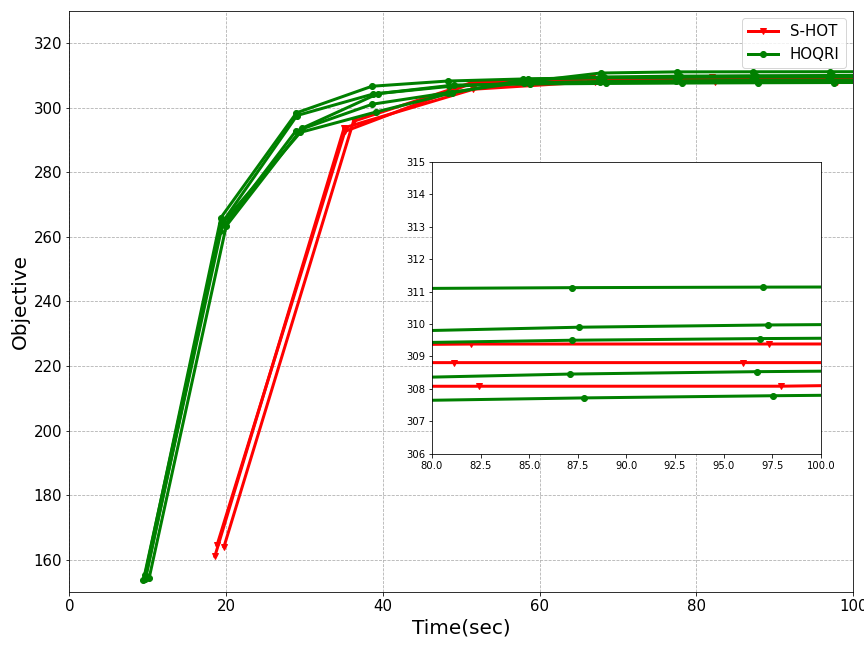}
    \caption{Rank $15\times25\times15$.} 
    \label{MovieLen152515}
\end{subfigure}
\caption{Convergence on the MovieLen data set.} 

\end{figure}

\begin{figure}[H]

\begin{subfigure}[t]{0.49\linewidth}  
    \centering 
    \includegraphics[width=\linewidth]{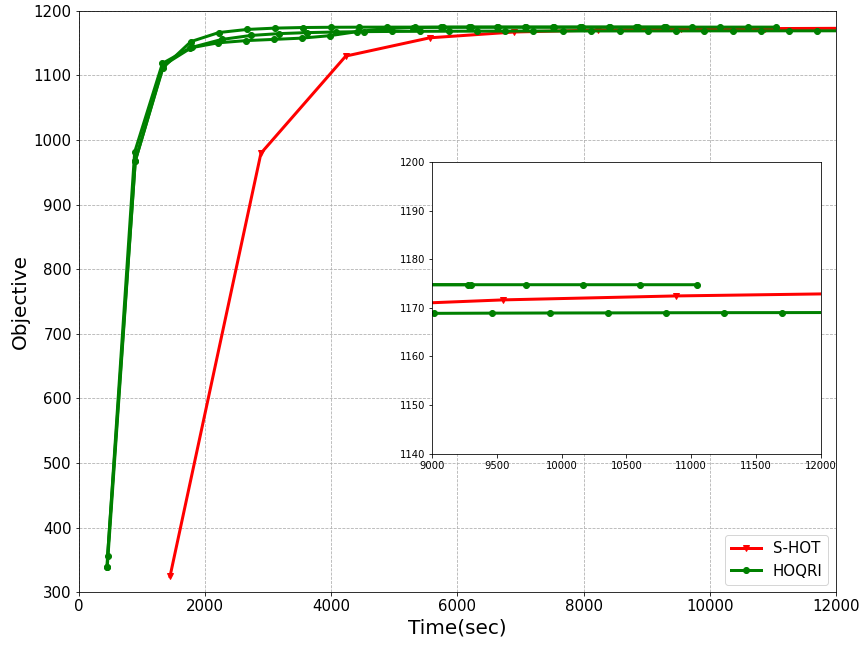}
    \caption{Rank $10\times10\times10$.} 
    \label{10M_MovieLen}
\end{subfigure}
\hfill
\begin{subfigure}[t]{0.49\linewidth}
    \centering
    \includegraphics[width=\linewidth]{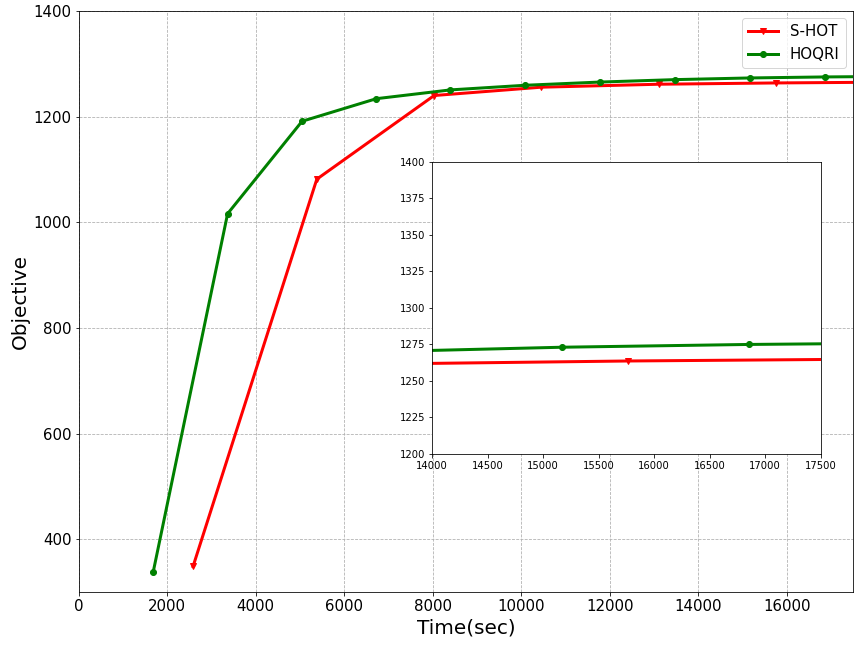}
    \caption{Rank $20\times20\times10$.} 
    \label{10M_MovieLen202010}
\end{subfigure}
\caption{Convergence on the 10M\_Movielen data set.} 
\end{figure}

\section{Conclusions}
\label{sec:conclusions}
In this paper, we proposed a new algorithm for the tensor Tucker decomposition called higher-order QR iteration (HOQRI). Compared to the celebrated higher-order orthogonal iteration (HOOI), HOQRI does not require a separate subroutine for the singular value decomposition. More importantly, HOQRI entirely resolves the intermediate memory explosion when computing the Tucker decomposition of very large and sparse tensors by defining a new tensor operation TTMcTC. Using concepts from manifold optimization, we show that HOQRI converges to a stationary point, which is a property that was not established for the celebrated HOOI.
Its outstanding performance is demonstrated on synthetic and real data in both MATLAB and C++ environment.

\appendix
\section{Supporting results for Section~\ref{sec:conv}}
\begin{lemma}\label{lmm:interlacing}
Denoting the singular values of $\bfA^\mode{n}_t$, in descending order, as $\sigma_1\geq\cdots\geq\sigma_{K_n}$ and the eigenvalues of $\bfG_{\mode{n}\,t}\bfG_{\mode{n}\,t}^\T$ as $\lambda_1\geq\cdots\geq\lambda_{K_n}$, then we have
\[
\sigma_k\geq\lambda_k, k=1,\ldots,K_n.
\]
\end{lemma}

We will be using the Cauchy interlacing theorem \cite[pp.118]{bellman1997introduction}, also known as the Poincar{\' e} separation theorem, stated as follows:
\begin{theorem}\label{thm:interlacing}
Let $\bfB\in\bbR^{n\times n}$ be a symmetric matrix with eigenvalues in descending order $\beta_1\geq\cdots\geq\beta_n$ and $\bfP\in\bbR^{n\times k}$ with orthonormal columns, i.e., $\bfP^\T\bfP=\bfI$. Denote the eigenvalues of $\bfP^\T\bfB\bfP\in\bbR^{k\times k}$ in descending order as $\gamma_1\geq\cdots\geq\gamma_k$, then for all $j\leq k$:
\[
\beta_j\geq\gamma_j\geq\beta_{n-k+j}.
\]
\end{theorem}

\begin{proof}[Proof of Lemma~\ref{lmm:interlacing}]
Consider the matrix $\bfA^\mode{n}_t\bfA^{\mode{n}\,\T}_t\in\bbR^{I_n\times I_n}$, it has $K_n$ nonzero eigenvalues $\sigma_1^2\geq\cdots\geq\sigma_{K_n}^2$ and $I_n-K_n$ zero eigenvalues. 

Given $\bfU^\mode{n}_t$ with orthonormal columns, we have
\begin{equation}\label{eq:UAAU}
\bfU^{\mode{n}\,\T}_t\bfA^\mode{n}_t\bfA^{\mode{n}\,\T}_t\bfU^\mode{n}_t
= \bfG_{\mode{n}\,t}\bfG_{\mode{n}\,t}^\T\bfG_{\mode{n}\,t}\bfG_{\mode{n}\,t}^\T
\in\bbR{K_n\times K_n}.
\end{equation}
As we have defined the eigenvalues of $\bfG_{\mode{n}\,t}\bfG_{\mode{n}\,t}^\T$, the eigenvalues of \eqref{eq:UAAU} in descending order are $\lambda_1^2\geq\cdots\geq\lambda_{K_n}^2$.

Applying the Cauchy interlacing theorem on the symmetric matrix $\bfA^\mode{n}_t\bfA^{\mode{n}\,\T}_t$ and orthonormal matrix $\bfU^\mode{n}_t$, we conclude that
\[
\sigma_k^2\geq\lambda_k^2, k=1,\ldots,K_n.
\]

Finally, singular values are by definition nonnegative, and the eigenvalues of $\bfG_{\mode{n}\,t}\bfG_{\mode{n}\,t}^\T$ are also nonnegative since it is positive semidefinite, this completes the proof of Lemma~\ref{lmm:interlacing}.
\end{proof}

\bibliographystyle{plain}
\bibliography{refs}
\end{document}